\documentclass[arxiv,reqno,twoside,a4paper,12pt]{amsart}
\usepackage{mltools}
\withdetails
\usepackage{mlmath62,mlthm10-REAR}

\usepackage{upref,amsmath,amssymb,amsthm,mathrsfs}
\usepackage{pgf,tikz, latexsym, tensor}
\usepackage{graphicx}
\usepackage[colorlinks=true,linkcolor=blue,citecolor=blue]{hyperref}

\usetikzlibrary{arrows}

\usepackage[scaled]{helvet} 
\usepackage{courier} 
\usepackage[mathbf]{euler}
\usepackage{mllocal} 

\usepackage{pgfplots}
\usepackage{xypic}

\numberwithin{equation}{section}

\definecolor{qqwuqq}{rgb}{0,0,0}

\begin{document}

\date{\today}

\title[On the domain of Dirac and Laplace Operators on Stratified Spaces]
{On the domain of Dirac and Laplace \\ type Operators on Stratified Spaces}

\author{Luiz Hartmann}
\address{Universidade Federal de S\~ao Carlos (UFSCar),
	Brazil}
\email{hartmann@dm.ufscar.br}
\urladdr{http://www.dm.ufscar.br/profs/hartmann}

\author{Matthias Lesch}
\address{Universit\"at Bonn,
	Germany}
\email{ml@matthiaslesch.de, lesch@math.uni-bonn.de}
\urladdr{www.matthiaslesch.de, www.math.uni-bonn.de/people/lesch}

\author{Boris Vertman} 
\address{Universit\"at Oldenburg, Germany} 
\email{boris.vertman@uni-oldenburg.de}

\thanks{Partial support by DFG Priority Programme "Geometry at Infinity", 
FAPESP: 2016/16949-8 and the Hausdorff Center 
for Mathematics, Bonn}


\begin{abstract}
We consider a generalized Dirac operator
on a compact stratified space with an iterated 
cone-edge metric. Assuming a spectral Witt condition, we prove its essential 
self-adjointness and identify its domain and the domain of its square with 
weighted edge Sobolev spaces. This sharpens 
previous results where the minimal domain is shown only to be a subset of an intersection 
of weighted edge Sobolev spaces. Our argument does not rely on microlocal techniques and 
is very explicit. The novelty of our approach is the use of an abstract functional analytic notion 
of interpolation scales. Our results hold for the Gauss-Bonnet and spin Dirac operators satisfying a 
spectral Witt condition.
\end{abstract}

\maketitle

\tableofcontents

\section{Introduction and statement of the main results}

Singular spaces arise naturally in various parts of mathematics.
Important examples of singular spaces include algebraic varieties and
various moduli spaces; singular spaces also appear naturally as
compactifications of smooth spaces or as limits of families of smooth
spaces under controlled degenerations. The development of analytic
techniques to study partial differential equations in the singular
setting is a central issue in modern geometry. 

\begin{figure}[h]
	\includegraphics[scale=0.75]{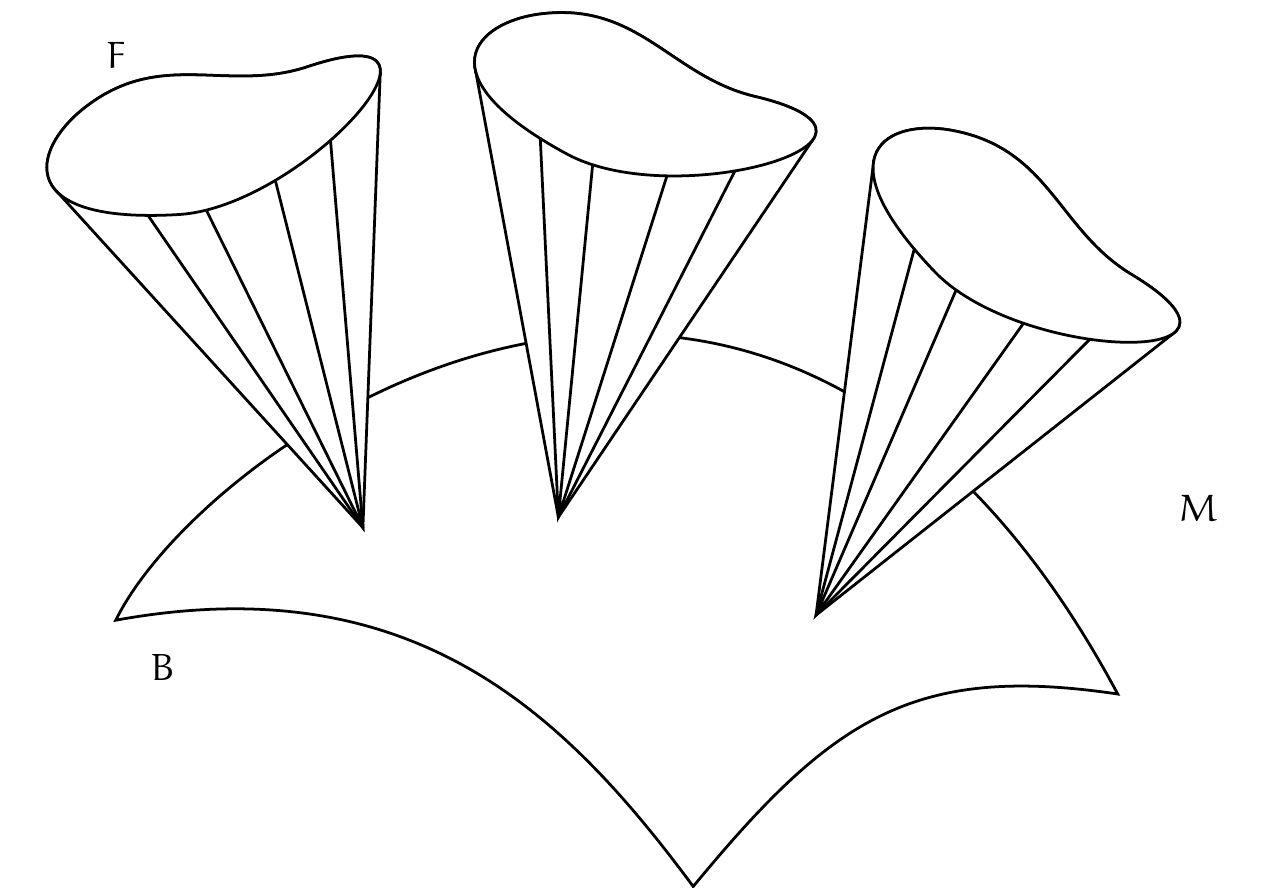}
	\caption{Simple Edge as a Cone bundle over $B$.}
	\label{figure1}
\end{figure}

Cheeger \cite{Che:SGS} was the first to initiate an influential program on
spectral analysis on smoothly stratified spaces with singular
Riemannian metrics.  Analysis of the associated geometric operators on
spaces with conical singularities was the focal point of
the research by Br\"uning and Seeley \cite{BruSee:RSA,
BruSee:TRE,BruSee}, Lesch \cite{Les:OOF}, Melrose \cite{Mel:TAP},
Schulze \cite{Schulze-cone1, Schulze-cone2}, Schrohe and Schulze \cite{Schulze-Schrohe1, 
Schulze-Schrohe2}, Gil, Krainer and Mendoza \cite{Krainer-cone, Krainer-cone2} to name just a few.  

Extensions to spaces with simple edge
singularities were developed by Mazzeo \cite{Maz:ETO}, as well as
Schulze \cite{Schulze, Schulze2} and his collaborators, see also 
Gil, Krainer and Mendoza \cite{Krainer-edge}. Various questions in spectral geometry and index theory 
on spaces with simple edge singularities have been addressed \eg 
by Br\"uning and Seeley \cite{BruSee}, Mazzeo and Vertman
\cite{MazVer, MaVe1}, Krainer and Mendoza \cite{Krainer-Mendoza, Krainer-Mendoza2}, 
Albin and Gell-Redman \cite{Albin-Jesse}, Piazza and Vertman \cite{PiVe}.

There have also been recent advances to lift the analysis to a very general setting of stratified spaces 
with iterated cone-edge singularities. Index theoretic questions for geometric Dirac operators on a general 
class of compact stratified Witt spaces with iterated cone-edge metrics have 
been studied 
by Albin, Leichtnam, Piazza and Mazzeo in \cite{ALMP1,ALMP2,Novikov}. The Yamabe problem on stratified spaces has been
solved by Akutagawa, Carron and Mazzeo in \cite{ACM}. 

If we wish to go a step further and do spectral geometry on stratified spaces, the crucial difficulty appears 
already in the setting of a stratified space of depth two, illustrated as in 
Figure \ref{figure2} below
with fibers $F_y$, at each $y \in B=Y_2$, being simple edge spaces. 
Consider \eg the family of Gauss--Bonnet operators on the fibers $F_y, y \in 
B$. 
Even if we impose a spectral Witt condition so that 
the Gauss--Bonnet 
operators on the fibers are essentially self-adjoint, their domains may 
still vary with the base point across $B$.
In case of variable domains however, smoothness of the operator 
family becomes a much more complicated issue, which needs to be resolved 
before any meaningful spectral geometric questions may be addressed. 

Our main result is formulated using the concept of a spectral Witt condition and the weighted
edge Sobolev space $\sH^{1,1}_e(M)$ on a stratified Witt space $M$ with an 
iterated 
cone-edge metric, which will be made explicit below. Elements of the edge 
Sobolev 
spaces take values in a Hermitian vector bundle $E$, which is suppressed from the notation.

For the moment, the spectral Witt condition is a spectral gap condition 
on certain operators on fibers $F$, see \Eqref{Witt-condition} and Definition 
\ref{Witt-stratified}, and in case 
of the 
Gauss--Bonnet operator on 
a stratified Witt space it can 
always be achieved by scaling the iterated cone-edge metric appropriately. The 
weighted edge 
Sobolev space $\sH^{s,\delta}_e(M)= \rho^\delta \sH^s_e(M)$ is the Sobolev 
space 
$\sH^s_e(M)$ of  all square integrable sections of the 
Hermitian vector bundle $E$ that remain square
integrable under weak application of $s\in \N$ edge vector fields, weighted with 
a $\delta$-th power of a smooth function $\rho$ that vanishes at
the singular strata to first order. Our main theorem is now as follows.

\begin{theorem}\label{main-intro}
	Let $M$ be a compact stratified space with an iterated cone-edge metric. 
	 Let $D$ denote either the Gauss--Bonnet or the spin Dirac operator, 
	 and assume the spectral Witt condition holds,\ie Definition 
	 \ref{Witt-stratified}. 
	 Then both $D$ and 
	 $D^2$ are essentially self-adjoint with domains
	\begin{equation}
	\begin{split}
	&\sD_{\max}(D) = \sD_{\min}(D) = \sH^{1,1}_e(M), \\
	&\sD_{\max}(D^2) = \sD_{\min}(D^2) = \sH^{2,2}_e(M).
	\end{split}
	\end{equation}
In case of the Gauss--Bonnet operator, sections take values in the exterior 
algebra $\gL^* ({}^{ie}T^*M)$
of the incomplete edge cotangent space $\gL^* ({}^{ie}T^*M)$. In case of the 
spin Dirac operator, sections
take values in the spinor bundle $S$.
\end{theorem}

Let us comment on related work in connection to Theorem \ref{main-intro}.
Gil, Krainer and Mendoza \cite[Theorem 4.2]{Krainer-edge} 
prove that for an elliptic differential wedge operator $A$ of order $m$ 
on a simple edge space, under an assumption on indicial roots, $\sD_{\min}(A)= \sH^{m,m}_e(M)$. Our theorem here extends
this statement to compact stratified spaces in the special case of the 
Gauss--Bonnet and Spin Dirac operators.
Moreover, Albin, Leichtnam, Piazza and Mazzeo in \cite[Prop.~5.9]{ALMP1} prove that under the spectral Witt condition 
the minimal domain $\sD_{\min}(D)$ of the Gauss--Bonnet operator is included in 
the intersection of 
$\sH^{1,\delta}_e(M)$ for all $\delta<1$. Our theorem here sharpens this statement into an equality instead of an inclusion.

In addition we emphasize that we employ different methods which are more elementary and do have a functional 
analytic flavor. Furthermore we also do not need singular pseudo-differential 
calculi.

\section{Smoothly stratified iterated edge spaces}\label{stratified-edge-spaces-section}

In this section we recall basic aspects of the definition of a compact smoothly stratified space 
of depth $k\in \N_0$, referring the reader for a complete discussion \eg  
to a very thorough analysis in \cite{ALMP1,ALMP2,Alb}. 

\subsection{Smoothly stratified iterated edge spaces of depth zero and one}

A compact stratified space of  depth $k=0$ is simply a compact Riemannian manifold.
A compact stratified space of depth $k=1$ is a compact simple 
edge space $\overline{M}$ with smooth open interior $M$, as discussed in \eg in \cite{Maz:ETO,MaVe1}. More precisely, $\overline{M}$ 
admits a single stratum $B\subset \overline{M}$ which is a smooth compact manifold. The 
edge $B$ comes with an open tubular neighborhood $\cU \subset \overline{M}$, a radial function  
$x$ defined on $\cU$, and a smooth fibration $\phi:\cU \to B$ with preimages 
$\phi^{-1}(q)\setminus\{q\}$, $q\in B$, being all diffeomorphic to 
open cones $C(F)=(0,1)\times F$ over a smooth compact manifold 
$F$. The restriction $x$ to each fiber $\phi^{-1}(q)$ is a radial function on that cone.
We also write $\phi: \b\cU\to B$ for the 
fibration of the $\{x=1\}$ level set over $B$. The tubular neighborhood $\cU 
\subset \overline{M}$ is illustrated in the Figure \ref{figure1}. 

The resolution $\widetilde{M}$ is defined by replacing the cones in the 
tubular neighborhood $\cU$ by 
finite cylinders $[0,1) \times F$. This defines a compact manifold with smooth boundary 
$\partial \widetilde{M}$ given by the total space of the fibration $\phi$. The resolution $\widetilde{\cU}$ of the singular neighborhood $\cU$ is defined 
analogously.

We equip the simple edge space with an edge metric $g$, which is smooth on 
$\overline{M}\setminus \cU$ and which over $\cU \backslash B$ takes the following form
\begin{equation}
g|_{\cU} = dx^2 + \phi^{\ast}g_B + x^2 g_F + h =: g_0 + h
\end{equation}
where $g_B$ is a Riemannian metric on $B$, $g_F$ is a 
smooth family of bilinear forms on the tangent bundle of the total space of the fibration $\phi:\b 
\cU \to B$, restricting to a Riemannian metric on fibers $F$, $h$ is smooth on 
$\widetilde{\cU}$ and 
$|h|_{g_0} = O(x)$, when $x\to 0$. We also require that $\phi: (\b\cU, g_F + 
\phi^{\ast}g_B)
\to (B,g_B)$ is a Riemannian submersion. 

Consider local coordinates $(x,y,\theta)$ on $\cU\backslash B \subset M$ near the edge, 
where $x$ is as before the radial coordinate, $y$ is
the lift of a local coordinate system on $B$ and $\theta$ restricts to local 
coordinates on each fiber $F$. Then, in terms of symmetric $2$-tensors 
$\textup{Sym}^2\{dx, x d\theta, dy\}$, generated by the $1$-tensors 
$\{dx, x d\theta, dy\}$, the higher order term $h$ satisfies over $\widetilde{\cU}$
\begin{equation}
h \in x \cdot C^\infty(\widetilde{\cU}, \textup{Sym}^2\{dx, x d\theta, dy\}).
\end{equation}

We finish with the standard definition of \emph{edge vector fields}.
The edge vector fields $\mathcal{V}_{e,1}$ are defined to be 
smooth on $\widetilde{M}$ and tangent to the fibers $F$ at $\partial 
\widetilde{M}$. 
We also write $\mathcal{V}_{ie,1} := x^{-1} \mathcal{V}_{e,1}$, which we call the \emph{incomplete edge}
vector fields. In the chosen local coordinate system $(x,y,\theta)$ we have explicitly
\begin{equation}
\begin{split}
&\mathcal{V}_{e,1}\restriction \widetilde{\cU} = C^\infty(\widetilde{\cU})\textup{- span}\, 
\{x\partial_x, x\partial_{y_1}, ..., x\partial_{y_{\dim B}}, \partial_{\theta_1},..., 
\partial_{\theta_{\dim F}}\}, \\
&\mathcal{V}_{ie,1}\restriction \widetilde{\cU} = C^\infty(\widetilde{\cU})\textup{- span}\, 
\{\partial_x, \partial_{y_1}, ..., \partial_{y_{\dim B}}, x^{-1}\partial_{\theta_1},..., 
x^{-1}\partial_{\theta_{\dim F}}\}.
\end{split}
\end{equation}

\subsection{Smoothly stratified iterated edge spaces of depth two}

A stratified space of depth $2$ is modelled as above but allowing the links 
$F$ to be stratified spaces of depth $1$, with 
smooth links. This is illustrated in Figure \ref{figure2}, and we proceed with studying 
this case in detail to provide a basis for a definition of smoothly stratified iterated edge spaces
of arbitrary depth. 
\begin{figure}[h]
	\includegraphics[scale=0.75]{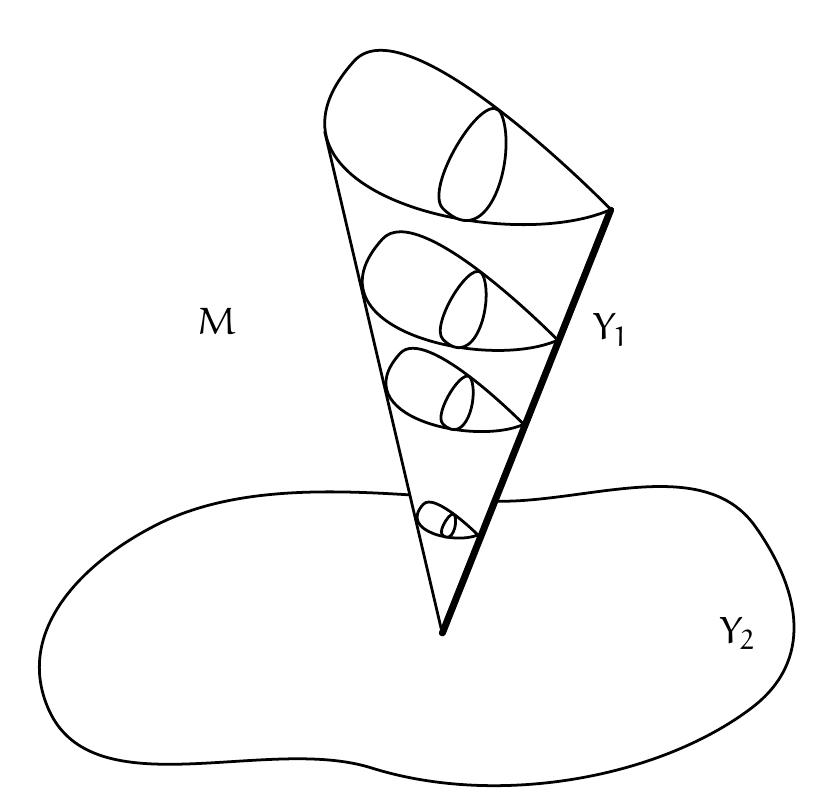}
	\caption{Tubular neighborhood $\cU\subset \overline{M}$, 
	$\overline{M}$ of depth $2$.}
	\label{figure2}
\end{figure}

The fibration of cones with singular links defines an open edge space itself with 
an open edge singularity in $Y_1$, which fibers over $Y_2$ and contains $Y_2$ in its closure. 
We now have two strata $\{Y_1,Y_2\}$ satisfying the following fundamental properties.
\begin{enumerate}
	\item[i)] $Y_2 \subset \overline{Y}_1$, and $Y_2$ is compact and smooth. 
	\item[ii)] Any point $q\in Y_1 = \overline{Y}_1\setminus Y_2$ has a tubular 
	neighborhood of cones with 
	smooth links. We say that $Y_1$ is a stratum of depth 1.
	Any point $q\in Y_2$ has a tubular neighborhood of cones
	$[0,1)\times F/_{(0, \theta_1) \sim (0, \theta_2)}$ with links $F$
	being stratified spaces of depth $1$. We say that $Y_2$ is a stratum of 
	depth $2$. 
	\item[iii)] We have the following sequence of inclusions 
	\begin{equation}
	\overline{M} \supset \overline{Y}_1 \supseteq Y_2 \supseteq 
	\varnothing.
	\end{equation}
    Then $\overline{M} \setminus 
	\overline{Y}_1$ is 
	an open Riemannian manifold dense in $\overline{M}$, and the strata of $\overline{M}$ are
	\begin{equation}
	Y_2,\;\; Y_1 = \overline{Y}_1\setminus Y_2,\;\; \overline{M}\setminus \overline{Y}_1.
	\end{equation}
\end{enumerate}

The resolution $\widetilde{M}$ is defined as in the depth one case by 
replacing the cones in the fibration $\phi: \cU \to Y_2$ by finite cylinders $[0,1)\times F$, and subsequently replacing the simple edge space $F$ with its resolution as well. This defines a compact manifold
with corners. The resolution $\widetilde{\cU}$ of $\cU$ is defined analogously. We denote the radial function on each cone in the fibration $\phi$ by $x$, and write $x'$ for the radial function of the simple edge space $F$. 

We can now define an iterated cone-edge metric $g$ as before by specifying
\begin{equation}
g|_{\cU} = dx^2 + \phi^{\ast} g_B + x^2 g_F + h=:g_0+h,
\end{equation}
where $B=Y_2$, $g_B$ is a smooth Riemannian metric, $g_F$ restricting on the links $F$ to 
iterated cone-edge metrics of depth $1$ (simple edge space). As before, these 
metrics 
$g_B$ and $g_F$
do not depend on the radial function $x$, and the higher order terms of the metric are included in 
the tensor $h$, which is smooth on $\widetilde{\cU}$ with $|h|_{g_0}= O(x)$ as $x\to 0$. We require that 
$\phi \restriction \b\cU: (\b\cU, g_F + \phi^{\ast}g_B)
\to (B,g_B)$ is a Riemannian submersion and put the same condition on the 
fibers $(F,g_F)$. 
 
The edge vector fields $\mathcal{V}_{e, 2}$, as well as the incomplete edge vector fields 
$\mathcal{V}_{ie, 2}$, are defined similarly to $\mathcal{V}_{e, 1}$ and  $\mathcal{V}_{ie, 1}$.
\begin{equation}\label{edge-vector-field-depth-2}
\begin{split}
&\mathcal{V}_{e,2}\restriction \widetilde{\cU} = C^\infty(\widetilde{\cU})\textup{- span}\, 
\{(xx')\partial_x, (xx')\partial_{y_1}, ..., (xx')\partial_{y_{\dim B}}, \mathcal{V}_{e, 1}(F)\}, \\
&\mathcal{V}_{ie,2}\restriction \widetilde{\cU} = C^\infty(\widetilde{\cU})\textup{- span}\, 
\{\partial_x, \partial_{y_1}, ..., \partial_{y_{\dim B}}, (xx')^{-1}\mathcal{V}_{e, 1}(F)\}.
\end{split}
\end{equation}
where $\mathcal{V}_{e, 1}(F)$ refers to the edge vector fields on the 
simple edge space $F$.

\subsection{Smoothly stratified iterated edge spaces of arbitrary depth}

At an informal level we can now say that $\overline{M}$ is a compact smoothly stratified iterated edge space
of arbitrary depth $k\geq 2$ with strata $\{Y_{\alpha}\}_{\alpha\in A}$ 
if $\overline{M}$ is compact and the following, inductively defined, properties are satisfied. 
     \begin{enumerate}
	\item[i)] If $Y_\alpha \cap \overline{Y}_\beta \neq \varnothing$ then 
	$Y_{\alpha}\subset \overline{Y}_\beta$ (each stratum is identified with its 
	open 
	interior).
	\item[ii)] The depth of a stratum $Y$ is the largest $(j-1) \in \N_0$ such 
	that there exists a chain of pairwise distinct strata $Y=Y_j,\; Y_{j-1},\ldots, Y_1$ with
	$Y_i \subset \overline{Y}_{i-1}$ for all $2	\leq i \leq j$. 
	\item[iii)] The stratum of maximal depth is smooth and compact. The maximal 
	depth of any stratum of 
	$\overline{M}$ is called the depth of $\overline{M}$.
	\item[iv)] Any point of $Y_\alpha$, a stratum of depth $j$, has a tubular 
	neighborhood of cones with links being stratified spaces of depth $j-1$, 
	for all $1\leq j \leq k$. 
	\item[v)] Setting 
	\begin{equation}
	\overline{M} = X_n \supset X_{n-1}=X_{n-2} \supseteq X_{n-3}\supseteq \cdots \supseteq 
	X_1 \supseteq X_0,
	\end{equation}
	where $X_j$ is the union of all strata of dimension less or equal than $j$,
	$X_n\setminus X_{n-2}$ is an open Riemannian manifold, dense in $\overline{M}$. 
\end{enumerate}
We call the union $X_{n-2}$ of all $Y_{\alpha}$, for 
$\alpha \in A$ the singular part of $\overline{M}$, and its complement in $\overline{M}$ the regular part, 
denoted by $M$. The precise definition of smoothly stratified spaces 
contains some other technical conditions, \cf Thom--Mather-spaces \cite{Alb}. 

The resolution $\widetilde{M}$ is a manifold with corners defined iteratively 
by resolving in each step the highest codimension singular strata as before. 
Each tubular neighborhood $\cU_\alpha$ of any point in $Y_\alpha$ admits 
a resolution $\widetilde{\cU}_\alpha$ in an analogous way. 

We define an iterated cone-edge metric $g$ on $M$ by asking $g$ to be an 
arbitrary smooth Riemannian metric away 
from singular strata, and requiring in each tubular neighborhood $\cU_\alpha$ of any point 
in $Y_\alpha$ to have the following form
\begin{equation}\label{iemetricUa}
g|_{\cU_\alpha} = dx^2 + \phi^{\ast}_\alpha g_{Y_\alpha} + x^2 g_{F_\alpha} + 
h=:g_0 + h,
\end{equation}
where $\phi_{\alpha}:\cU_\alpha \to \phi_{\alpha} (\cU_\alpha) \subseteq Y_\alpha$ is the 
obvious fibration, $\phi_{\alpha} (\cU_\alpha)$ is open in $Y_\alpha$, the restriction 
$g_{Y_\alpha} \restriction \phi_{\alpha} (\cU_\alpha)$ is a smooth Riemannian metric,
$g_F$ is a symmetric two tensor on the level set $\{x=1\}$, whose restriction to the links
$F_\alpha$ (smoothly stratified iterated edge spaces of depth at most $(k-1)$) 
is an iterated cone-edge metric. The higher order term $h$ is smooth on 
$\widetilde{\cU}_\alpha$
and satisfies $|h|_{g_0}=O(x)$, when $x\to 0$.
We also assume that $\phi_\alpha \restriction \b\cU_\alpha: (\b\cU_\alpha, g_{F_\alpha} + \phi^{\ast}_\alpha g_{Y_\alpha})
\to (\phi_{\alpha} (\cU_\alpha), g_{Y_\alpha})$ is a Riemannian submersion and 
put the same condition in the lower depth. Existence of such 
smooth iterated cone-edge metrics is discussed in \cite[Prop.~3.1]{ALMP1}.

The definition of edge vector fields
$\mathcal{V}_{e, k}$ and incomplete edge vector fields $\mathcal{V}_{ie, k}$, extends to the smoothly stratified space $M$
by an inductive procedure as in case of $k=2$, cf. \eqref{edge-vector-field-depth-2}. To be precise, 
denote by $\rho$ a smooth
	function on the resolution $\widetilde{M}$, nowhere vanishing
	in its open interior, and vanishing to first order at each 
	boundary face. Then $\mathcal{V}_{e,k} = \rho \mathcal{V}_{ie,k}$
	and
	\begin{equation}\label{edge-vector-field-depth-k}
	\begin{split}
	&\mathcal{V}_{e,k}\restriction \widetilde{\cU} = C^\infty(\widetilde{\cU})\textup{- span}\, 
	\{\rho \partial_x, \rho\partial_{s_1}, ..., 
	\rho \partial_{s_{\dim Y_\alpha}}, \mathcal{V}_{e, k-1}(F_\alpha)\}, \\
	&\mathcal{V}_{ie,k}\restriction \widetilde{\cU} = C^\infty(\widetilde{\cU})\textup{- span}\, 
	\{\partial_x, \partial_{s_1}, ..., \partial_{s_{\dim Y_\alpha}},  \rho^{-1}\mathcal{V}_{e,k-1}(F_\alpha)\}.
	\end{split}
	\end{equation}

\subsection{Sobolev spaces on smoothly stratified iterated edge spaces}

We may now define the edge Sobolev spaces in the setup of 
a compact stratified space $M$ of depth $k$ with an iterated cone-edge metric.
Let ${}^{ie}TM$ denote the canonical vector bundle defined by the condition that the 
incomplete edge vector fields $\mathcal{V}_{ie, k}$ form locally a spanning set of sections
$\mathcal{V}_{ie, k} = C^\infty(M, {}^{ie}TM)$. We denote by ${}^{ie}T^*M$ the dual of ${}^{ie}TM$,
also referred to as the incomplete edge cotangent bundle. We write $E=\Lambda^* ({}^{ie}T^*M)$, when 
discussing the Gauss--Bonnet operator, and we set $E$ to be the spinor bundle, 
when 
discussing the spin Dirac operator. In either of these cases we define the edge Sobolev 
spaces with values in $E$ as follows.

\begin{definition}\label{Sobolev-spaces}
Let $M$ be a compact smoothly stratified iterated edge space of arbitrary depth $k\in \N$
with an iterated cone-edge metric $g$. We denote by $L^2(M, E)$ the $L^2$ 
completion of smooth compactly supported differential forms $C^\infty_0(M, E)$. Denote by $\rho$ a smooth
function on the resolution $\widetilde{M}$, nowhere vanishing
in its open interior, and vanishing to first order at each 
boundary face. Then, for any $s\in \N$ and $\delta \in \R$
we define the weighted edge Sobolev spaces by
\begin{equation}
\begin{split}
&\sH_e^s(M):= \{\w \in L^2(M) \mid V_1 \circ \cdots \circ V_s \w \in L^2(M,E), \ 
\textup{for} \ V_j \in \mathcal{V}_{e, k}\}, \\
&\sH_e^{s, \delta}(M):= \{\w = \rho^\delta u \mid u \in \sH_e^s(M)\},
\end{split}
\end{equation}
where $V_1 \circ \cdots \circ V_s \w \in L^2(M,E)$ is understood in the distributional sense\footnote{
This is not the ordinary Sobolev space $H^s(\R_+)$ if $M=\R_+$. }.
\end{definition}

\section{Interpolation scales of Hilbert Spaces}
\label{s.ISHS}

\subsection{Preliminaries}
Let $H_1, H_2$ be Hilbert spaces which are assumed to be embedded into a
barrelled locally convex topological vector space, such that it makes sense to
talk about $H_1+H_2$ (non-direct sum space) and $H_1\cap H_2$.  Let
$[H_1,H_2]_{\theta}$, $0\leq \theta \leq 1$, be their complex interpolation
space. For Calder\'{o}n's complex interpolation theory \cite{Cal1964} we refer to
\cite[Sec.~4.2]{Tay}.  The space of bounded linear operators between $H_1, H_2$
is denoted by $\sL(H_1,H_2)$, resp. if $H_1=H_2=H$ we just write $\sL(H)$.

If $H_2 \hookrightarrow H_1$ is densely embedded such that the norm of $H_2$ is
the graph norm of the nonnegative self-adjoint operator $\gL$ in $H_1$,
then by \cite[Prop.~2.2]{Tay}
\begin{equation}\label{eq.ISHS.1}
[H_1,H_2]_{\theta} = \sD (\gL^{\theta}).
\end{equation} 
In fact, there is a converse to this statement.

\begin{lemma}\label{p.ISHS.1}
Let $T: H_1 \to H_2$ be a bounded operator between Hilbert spaces $H_1$ and 
$H_2$. Then we have the equality of ranges $\ran T = \ran \sqrt{T\ T^*}$.
\end{lemma}

\begin{proof}
  Let $T^* = U|T^*| = U\sqrt{T\ T^\ast}$ be the polar decomposition of $T^*$;
$U$ is a partial isometry with $\ran U = \ovl{\ran T^*} = (\ker T)^\perp$ 
and $\ker U = \ker T^* = (\ran T)^\perp$.
Then, taking adjoints $T = \sqrt{T\ T^*} \ U^*$, and hence 
$\ran \sqrt{T \ T^\ast}\supset \ran T$. Since 
$\ran U^* = (\ker  U)^\perp = (\ker T^*)^\perp = (\ker \sqrt{T \ T^*})^\perp$, 
the equality follows.
\end{proof}

\begin{proposition}[{\cite[Sec.~I.2.1]{LioMag1972}}]
  \label{p.ISHS.2}      
  Let $H$ be a Hilbert space with a dense subspace $\sD\subset H$.
Assume  that $\sD$ carries a Hilbert space structure such that the inclusion
map $i:\sD \hookrightarrow H$ is continuous. Then $\sD = \ran \sqrt{i \ i^*}$ and
$\sqrt{i \ i^*}: H \to \sD$ is a unitary isomorphism. $\gL:=(\sqrt{i \
  i^*})^{-1}$ is a self-adjoint operator with domain $\sD$, hence
\begin{equation}\label{eq.ISHS.2} 
   [H,\sD]_\theta = \sD(\gL^\theta),\;\theta\in[0,1].
\end{equation}	
\end{proposition}

\begin{proof} By \eqref{eq.ISHS.1}, see \cite[Proposition 2.2]{Tay}, the last claim follows once the claims about the operator $\gL$ are established.
From Lemma \ref{p.ISHS.1} we know that $\sD = \ran\sqrt{i \ i^*}$. Note 
that $\ker i = \{0\}$, $\ovl{\ran i} = H$ and hence $i^*$ and
$\sqrt{i \ i^*}$ are injective with dense range. 
Consequently, $\gL = (\sqrt{i \ i^*})^{-1}$ is self-adjoint with domain $\sD$.
For $y\in \sD$ we find
\begin{equation}\label{eq.ISHS.3}
\begin{aligned}
\| \sqrt{i\ i^*} \ y\|_{\sD}^2 &=\langle \sqrt{i\ i^*} \ y, \sqrt{i\ i^*} \ 
	y\rangle_\sD = \langle  i^* \gL\ y, \sqrt{i\ i^*} \ y\rangle_\sD\\
	&=\langle  \gL\ y, i\ \sqrt{i\ i^*} \ y\rangle_H 
          =\langle   y, \gL \ \sqrt{i\ i^*} \ y\rangle_H = \|y\|^2_H.
\end{aligned}
\end{equation}
Since $\sD$ is dense the claim follows.
\end{proof}

\subsection{Scales of Hilbert Spaces}

From Br\"uning and Lesch \cite[Section 2]{BL1} we recall the useful concept
of a scale of Hilbert spaces, which has been used in various forms by
several authors, see Connes and Moscovici \cite[Appendix B]{CoMo95}, Higson
\cite[\S 4]{Hig06}, Otgonbayar \cite{Otg09} and Paycha \cite{Pay10}. Let $H$ 
be a Hilbert space and $A$ a self-adjoint operator in $H$. Then
\begin{equation}\label{eq.ISHS.4}
H^{\infty}:= \bigcap_{n=0}^\infty \sD (|A|^n) = \bigcap_{n=0}^\infty \sD(A^n)
\end{equation}
is dense in $H$. For $s\in \R$, let $H^s(A)$ be the completion of $H^\infty$ 
with respect to the scalar product
\begin{equation}\label{eq.ISHS.5}
\langle x , y \rangle_s:=\langle (I+ A^2)^{\frac{s}{2}} \ x , 
(I+A^2)^{\frac{s}{2}}\ y \rangle.
\end{equation}  
Then $H^n(A) = \sD(A^n) = \sD(|A|^n)$  for $n\in \Z_+$, respectively, $H^s(A) = 
\sD(|A|^s)$, for any $s\geq 0$. 

The properties of the family $\{H^s\}_{s\in \R}$ are reminiscent of properties of 
Sobolev spaces and they are summarized in the following proposition. 

\begin{proposition}\label{p.ISHS.3} 
The family $(H^s(A))_{s\geq 0}$ satisfies:
\begin{thmenum}
\item $H^s$ is a Hilbert space, for all $s\geq 0$.
		
\item For $s'\geq s$ we have a continuous embedding
$H^{s'}\hookrightarrow H^s$.

\item $[H_s,H_t]_\theta = H_{\theta t + (1-\theta) s }$, for 
$0\leq \theta \leq 1$, \\ 
in the sense of complex interpolation.
	
\item $H^\infty= \bigcap\limits_{s\geq 0} H^s$ 
is dense in $H^t$ for each $t$. 
\end{thmenum}
\end{proposition}

An abstract family $(H^s)_{s\geq 0}$ of Hilbert spaces satisfying
(1)--(4) is called an (interpolation) scale of Hilbert spaces.
If there exists a self-adjoint operator $A$ such that $H^s=H^s(A)$
for $s\ge0$ then we call $A$ a generator of the scale.
The item (4) implies that $H^{s'}$ is dense in $H^s$ for $s'\geq s$. 
Proposition \ref{p.ISHS.2} implies that for $N>0$ there exists a 
self-adjoint operator $\gL\geq0$ with $\sD(\gL^N) = H^N$ 
and hence $H^s = \sD(\gL^s) =: H^s(\gL)$ for
$0\leq s\leq N$.

\begin{remark}\label{p.ISHS.4} 
Given a scale $(H^s)_{s\geq 0}$ of Hilbert spaces as in Proposition \ref{p.ISHS.3},
one may ask whether there exists a generator $\gL$,  such that
$H^s = \sD(\gL^s)$ for all $s\geq0$, and not only for $0\leq s\leq N$.

We believe that in general the answer is no. \textit{E.g.,} the scale of 
Sobolev spaces
  $H^s([0,\infty))$ does not have a natural generator, although we cannot prove
  that there does not exist one.  We leave this open question to the reader.
  This does not affect the discussion below.
\end{remark}

Nonetheless, in the sequel we will for convenience assume that the scales do
have a global generator $\gL$. As the arguments will always only concern a
compact set of $s$-values, in light of the discussion above,
this is not really a loss of generality. 

Thus for all practical purposes we may think of a Hilbert space scale being 
the scale of a positive operator $\gL$.
We note that if two positive self-adjoint operators  $\gL_1$, 
$\gL_2$ have the same domain $\sD(\gL_1) = \sD(\gL_2)$
then $H^1(\gL_1)=H^1(\gL_2)$, and by complex interpolation
\begin{equation}\label{eq.ISHS.6}
  \sD(\gL_1^s) = [H,H^1(\gL_1)= H^2(\gL_2)]_s = \sD(\gL_2^s),\;{\text{for}}\; 
0\leq s \leq 1.
\end{equation}
In general, however, we will have 
\begin{equation}\label{eq.ISHS.7}
\sD(\gL_1^s)\not=\sD(\gL_2^s),\;\text{for}\; s>1.
\end{equation}

\begin{example}\label{p.ISHS.5}   
To illustrate this by example consider
\begin{equation}\label{eq.ISHS.8}
\gL_1 := \left(\begin{array}{cc}
0 & \partial_x\\-\partial_x & 0
\end{array}\right), \quad 
\gL_2 := \left(\begin{array}{cc}
0 & \partial_x+a\\-\partial_x+a & 0
\end{array}\right),
\end{equation}
acting in the Hilbert space $L^2(\R_+,\C^2)=L^2(\R_+)\otimes \C^2$
with domain
\begin{equation}\label{eq.ISHS.9}
\sD(\gL_1) = \sD(\gL_2)
  = \bigsetdef{f={f_1\choose f_2}\in H^1(\R_+)\otimes \C^2}{f_1(0)=0}. 
\end{equation} 
It is straightforward to see that $\gL_j, j=1,2$ are
self-adjoint. However, the domains of the squares are given
by
\begin{equation}\label{eq.ISHS.10}
\begin{split}
  \sD(\gL_1^2) &= \bigsetdef{f\in H^2(\R_+)\otimes \C^2}{ f_1(0) =0,\ 
  f_2'(0) = 0  }, \\
  \sD(\gL_2^2) &= \bigsetdef{f\in H^2(\R_+)\otimes \C^2 }{ f_1(0) =0,\ 
f_2'(0)+a\cdot f_2(0) = 0 },
\end{split}
\end{equation} 
thus $H^s(\gL_1) \not= H^s(\gL_2)$ for $1<s\le 2$.
\end{example}

In view of Example \ref{p.ISHS.5}, we may now ask for criteria such that
two self-adjoint operators generate the same interpolation scale. 
 
\begin{definition}\label{p.ISHS.6} 
Let $\gL$ be a self-adjoint operator in the Hilbert space $H$ with
interpolation scale $H^s(\gL)_{s\geq 0}$.
A linear operator $P:H^\infty (\gL) \to H^\infty(\gL)$ is said to be
of order $\mu$ if $P$ admits a formal adjoint\footnote{This means 
  that there is $P^t:H^\infty\to H^\infty$ such that for all $x,y\in H^\infty$, 
  $\inn{Px,y} = \inn{x,P^ty}$.} with respect to the scalar product of $H$,
and for any $s\in \R$, $P$ and $P^t$ extend by continuity
$H^s(\gL) \to H^{s-\mu}(\gL)$. I.~e. there are constants $C_s(P), C_s(P^t)$
such that for $x\in H^{\infty}$ we have $\|Px\|_{s-\mu} \le C_s(P)\cdot
\|x\|_s$ and $\|P^tx\|_{s-\mu} \le C_s(P)\cdot \|x\|_s$. 
By $\Op^\mu( \gL )$ we denote the operators of order $\mu$.
\end{definition}

Clearly, $\Op^\bullet(\gL) = \bigcup_\mu \Op^\mu(\gL)$ is a filtered
algebra of operators acting on $H^\infty(\gL)$. \mpar{ Ref to Higson, Connes maybe }
To show that an operator $P$ is of order $\mu$ it suffices to check
the estimates in the definition on a sequence $(t_j)_j$ of $t$-values
with $\lim t_j = \infty$. This follows again from complex interpolation.

The continuity condition can equivalently be formulated in terms of the
resolvent of $\gL$:
\begin{equation}\label{eq.ISHS.11}
\xymatrix@-1.75pc{H^t(\gL) \ar[rr]^{P} \ar[dd]_{(I+|\gL|)^t} & & 
H^{t-\mu}(\gL)\ar[dd]^{(I+|\gL|)^{t-\mu}}\\
	& \circlearrowleft & \\
	H=H^0(\gL)\ar[rr] &  & H^0(\gL)=H.}
\end{equation}
Here, the lower arrow is given by the operator
\begin{equation}\label{eq.ISHS.12}
(I+|\gL|)^{t-\mu} \circ P \circ (I+|\gL|)^{-t},
\end{equation}
which is required to be bounded on $H$ for all $t\in \R$.
If $P=\gL_2$ is a selfadjoint operator of order $1$ 
on the Sobolev-scale $H^\bullet(\gL_1)$, then we have an equality of
interpolation scales $H^\bullet(\gL_1) = H^\bullet(\gL_2)$, and hence
we conclude using the interpolation property with the following observation.

\begin{proposition}\label{p.ISHS.7}   
  Assume that for any $n\in\N\setminus \{0\}$
  \begin{equation}\label{eq.ISHS.13} 
  (I+|\gL_1|)^{n-1} \circ \gL_2 \circ (I+|\gL_1|)^{-n}
  \end{equation}
  is bounded on $H$. Then $\gL_1$ and $\gL_2$ generate the same 
  interpolation scales. If \eqref{eq.ISHS.13} is bounded only for $n=1$ then 
  we can only infer that  
  $H^s(\gL_1)= H^s(\gL_2)$ for $0 \leq s \leq 1$.
\end{proposition}

\subsection{Tensor products of interpolation scales}
In this section we follow in part \cite[Sec.~2]{BL1}.
We fix two interpolation scales $\{H_j^s\}_{s\geq 0}$, $j=1,2$ with generators
$\gL_1$, $\gL_2$. Without loss of generality we may choose $\gL_1$,
$\gL_2$ such that they are greater or equal to $I$ and hence we may define
the scalar product on $H_j^s$ by
\begin{equation}\label{eq.ISHS.14}
\langle x,y\rangle_{H^s_j}:=\langle \gL_j^s x, \gL_j^s y \rangle_{H_j}.
\end{equation}

For tensor products of (unbounded) operators we refer to the
Appendix \ref{s.TPO}, in particular 
Proposition \ref{p.TPO.2}. $H_1\hot  H_2$ resp. $A\hot  B$
denotes the completed Hilbert space tensor product resp. the tensor product
of (unbounded) operators $A,B$.

\begin{lemma}\label{p.ISHS.8}    
$\{H_1^s \hot  H^s_2\}_{s\geq 0}$ is an interpolation scale with generator 
$\gL_1 \hat\otimes \gL_2$. 
\end{lemma}

\begin{proof}
By Proposition \ref{p.TPO.3}, we have
$\gL_1\hot \gL_2 \geq I$, hence the graph norm of 
$(\gL_1\hot \gL_2)^s$ is equivalent to
$\|(\gL_1\hot \gL_2)^s x\|$. Note furthermore,
that  $\gL_1\hot I$ and $I\hot \gL_2$ are commuting
self-adjoint operators greater or equal to $I$, thus 
\begin{equation}\label{eq.ISHS.15}
(\gL_1\hot \gL_2)^s 
  = (\gL_1\hot  I \cdot I\hot \gL_2)^s 
  = \gL^s_1\hot  I \cdot I\hot \gL_2^s
  =\gL^s_1\hot  \gL_2^s.
\end{equation}
Furthermore, for $x_j\in H_1^\infty$, $y_j\in H_2^\infty, j=1,\ldots,r$
we have with each summation index running from $j=1,\ldots,r$:
\begin{equation*}
\begin{aligned}
\Bigl \|\sum_j x_j\otimes y_j\Bigr\|^2_{H_1^s\hot  H_2^s} 
  &=  \sum_{k,l}\langle x_k \otimes y_k, x_l\otimes y_l \rangle_{H_1^s \hot H_2^s}\\
  &= \sum_{k,l}\langle x_k, x_l\rangle_{H_1^s} \langle y_k, y_l \rangle_{H_2^s}\\
  &= \sum_{k,l}\langle \gL_1^s x_k, \gL_1^s x_l\rangle_{H_1} \langle 
  \gL_2^sy_k, \gL_2^s y_l \rangle_{H_2}\\
  &= \Bigl \langle \gL_1^s \hot \gL_2^s \Bl\sum_j x_j\ot y_j\Br, 
      \gL_1^s \hot \gL_2^s \Bl\sum_j x_j\ot y_j\Br \Bigr\rangle_{H_1\otimes 
      H_2}.
\end{aligned}
\end{equation*}
This shows that the tensor product norm on $H_1^s\hot H_2^s$ is equivalent
to the graph norm of $\gL_1^s\hot \gL_2^s$ which proves the claim.
\end{proof}

As a consequence we get for $s,t\geq 0$
\begin{equation}\label{eq.ISHS.17}
\bigl [H_1^s\hot  H_2^s, H_1^t \hot  H_2^t\bigr ]_\theta = 
    H_1^{\theta t + (1-\theta s )} \hot  H_2^{\theta t + (1-\theta)s}
              =[H_1^s,H_1^t]_\theta \hot [H_2^s,H_2^t]_\theta .
\end{equation}
Since every interpolation pair of Hilbert spaces may be embedded into an 
interpolation scale (Proposition \ref{p.ISHS.2}) we obtain 

\begin{corollary}\label{p.ISHS.9} 
If $E'\subset E$, $F'\subset E$ are interpolation pairs of Hilbert spaces 
then, for $0\leq \theta \leq 1$,
\begin{equation}\label{eq.ISHS.18}
    [E\hot F,E'\hot F']_{\theta} = 
    [E,E']_\theta\hot [F,F']_\theta.
\end{equation}
\end{corollary}

\begin{remark}\label{p.ISHS.10}
The tensor product of Lemma \ref{p.ISHS.8} should not be 
confused with the Sobolev spaces on product spaces. Note that on $\R^n$ we 
have $H^s(\R^n) = H^s(\Delta_{\R^n})$, where $\Delta_{\R^n} = - \sum_{j=1}^n \partial_{x_j}^2$
is the Laplace operator. Now it is not true that 
\begin{equation}\label{eq.ISHS.19}
   H^s(\R^n \times \R^m) = H^s(\R^n)\hot H^s(\R^m).
\end{equation}
Rather we have the following equalities
\begin{equation}\label{eq.ISHS.20}
\begin{aligned}
  H^s(\R^n \times \R^m) 
     &= H^s(\Delta_{\R^n\times\R^m} = \Delta_{\R^n}\hot  I + I \hot \Delta_{\R^m})\\
     &\stackrel{!}{=} H^s(\Delta_{\R^n})\hot L^2(\R^m) \ \cap \ 
        L^2(\R^n)\hot  H^s(\R^m), \;\text{ for } s\geq 0.
\end{aligned}
\end{equation}
This is due to the equality of domains
\begin{equation}\label{eq.ISHS.21}
   \sD((\Delta_{\R^n}\hot  I + I 
	\hot \Delta_{\R^m})^s)=\sD(\Delta^s_{\R^n}\hot  I)\cap 
	\sD(I \hot \Delta^s_{\R^m}),
\end{equation}
as we will see below.
\end{remark}

\begin{definition}\label{p.ISHS.11}
Given two interpolation scales $\{H^s_j\}_{s\geq 0}, j=1,2$, 
we put for $s\geq 0$
\begin{equation}\label{eq.ISHS.22}
   \sH^s:= \sH^s(\{ H_1^{\bullet} \},\{ H_2^{\bullet} \})
	:=H_1^s \hot  H_2^0 \cap H_1^0 \hot  H_2^s.
\end{equation}
\end{definition}
This is a Hilbert space with scalar product being the sum of the scalar
products of $H_1^s\hot H_2^0$ and $H_1^0 \hot  H_2^s$.

\begin{proposition}\label{p.ISHS.12}
Let $\gL_1,\gL_2 \geq I$ be generators of 
$\{H_1^\bullet\}$, $\{H_2^\bullet\}$, respectively. 
Then $\{\sH^s\}_{s\geq 0}$ is an 
interpolation scale with generator $\gL_1\hot I + I \hot  \gL_2$ and
\begin{equation}\label{eq.ISHS.23}
  \sH^s = \bigcap_{0\leq t \leq s} H_1^t \hot  H_2^{s-t}
  = \left(H_1^s \hot  H_2^{0}\right) \cap \left(H_1^0 \hot  H_2^{s}\right).
\end{equation}
\end{proposition}

\begin{proof}
Recall that $\gL_1\hot I$, $I\hot \gL_2$ are 
commuting self-adjoint operators greater or equal to $I$. Now from 
\begin{equation}\label{eq.ISHS.24}
  \frac{1}{2} (b^s+c^s) \leq (b+c)^s\leq 2^s (b^s+c^s),
\end{equation}
for $b,c,s\geq 0$ and the Spectral Theorem we infer
\begin{equation}\label{eq.ISHS.25}
    \sH^s = \sD(\gL_1^s \hot  I) \cap \sD(I \hot  
   	\gL_2^s) = \sD((\gL_1\hot I + I \hot \gL_2)^s),
\end{equation}
hence the first part of the statement follows. 
	
For the second part, we first note
that the concavity of the $\log$--function implies the inequality
\begin{equation}\label{eq.ISHS.26}
  a^{\theta}\cdot b^{1-\theta} \leq \theta \cdot a + (1-\theta)\cdot b,
\end{equation}
for $a,b\geq 0$ and $0 \leq \theta \leq 1$. For the commuting
operators $\gL_1\hot I$, $I\hot \gL_2$ and $0\leq \theta \leq 1$
the inequality implies
\begin{equation}\label{eq.ISHS.27}
  \begin{aligned}
    H_1^0\hot H_2^s \cap H_1^s\hot H_2^0 
      &= \sD(I\hot \gL_2^s) \cap \sD(\gL_1^s\hot I)\\
      &\subset \sD((\gL_1^s\hot I)^\theta \cdot	(I\hot \gL_2^s)^{1-\theta})\\
      &=\sD(\gL_1^{\theta s} \hot  \gL_2^{s-\theta s})
        =H_1^{\theta s} \hot  H_2^{s-\theta s}.
  \end{aligned}
\end{equation}
Consequently, $\sH^s \subset \bigcap_{0\leq t \leq s} H_1^t \hot  H_2^{s-t}$.
\end{proof}

\section{Dirac operators on an abstract edge}\label{abstract-Dirac-section}

\subsection{Generalized Dirac operators on an abstract edge}

Let $S$ be a smooth family of self-adjoint operators in a Hilbert space $H$ 
with parameter $y\in \R^b$ and a fixed domain $\sD_S$. We assume that each 
$S(y)$ is discrete. A generalized Dirac Operator $D$ acting on 
$C_0^\infty(\R_+\times\R^b,H^\infty)$ is defined by the following (differential) expression
\begin{equation}\label{generalized-Dirac}
D:=\Gamma (\partial_x + X^{-1}  S)+ T,
\end{equation}
where $x\in \R_+$, $X$ denotes the multiplication operator by $X$, $\Gamma$ is 
skew-adjoint and a unitary operator on the Hilbert space $L^2(\R_+\times 
\R^b,H)$, and $T$ is a symmetric generalized Dirac Operator on $\R^b$, given in 
terms of coordinates $(y_1,\ldots,y_b)\in\R^b$ and smooth families $(c_1(y), c_b(y))$
of bounded linear operators on $H$, which satisfy Clifford relations for 
each fixed $y\in \R^b$, by
\begin{equation}
T= \sum_{j=1}^{b}c_j(y) \frac{\partial}{\partial y_j}.
\end{equation}
Here, we have hid the vector bundle value action of the Dirac Operator $T$ into 
the Hilbert space $H$. 

We assume that the following standard commutator relations hold
\begin{equation}\label{Eq1}
\begin{aligned}
\Gamma \ S + S \ \Gamma &= 0,\\
\Gamma \ T + T \ \Gamma &= 0,\\
T \ S - S \ T &=0.
\end{aligned}
\end{equation}
In \S \ref{s.AEDSE} we show that the Gauss--Bonnet 
operator on a simple edge satisfies these relations, cf. \eqref{commutator-GB}.
The same relations hold for the spin Dirac operator, as shown in 
\cite[(3.16), (3.18)]{Albin-Jesse}. 

We shall also consider $D$ with coefficients frozen at some $y_0 \in \R^b$
\begin{equation}
D_{y_0}:=\Gamma (\partial_x + X^{-1} S(y_0))+ T_{y_0}, \quad 
\textup{where} \ T_{y_0}= \sum_{j=1}^{b}c_j(y_0) \frac{\partial}{\partial y_j}.
\end{equation}
Consider the Fourier transform $\mathfrak{F}_{y\to\xi}$ on the
$L^2(\R^b)$-component of $L^2(\R_+\times\R^b,H)$. We use H\"ormander's normalization and write 
\begin{equation}
\left(\mathfrak{F}_{y\to\xi} f\right)(\xi) = 
\int_{\R^b} e^{-i \langle y, \xi \rangle} f(y) dy, 
\quad \left(\mathfrak{F}^{-1}_{y\to\xi} g\right)(y) = 
\int_{\R^b} e^{i \langle y, \xi \rangle} g(\xi) \frac{d\xi}{(2\pi)^b}. 
\end{equation}
We compute
\begin{equation}\label{L-transform}
\mathfrak{F}_{y\to\xi} \circ D_{y_0} \circ \mathfrak{F}_{y\to\xi}^{-1} = \Gamma 
(\partial_x + X^{-1}S(y_0)) + ic(\xi; y_0) =: L(y_0,\xi),
\end{equation}
where
\begin{equation}
c(\xi; y_0) := \sum_{j=1}^b c_j(y_0) \xi_j.
\end{equation}
The usual strategy is now to study invertibility of $L(y_0,\xi)$ on appropriate 
spaces, which is then used to construct the parametrix for $D$ and analysis 
of its domain.

\subsection{The spectral Witt condition}

We also impose a {\it spectral Witt condition}, which asserts that 
\begin{equation}\label{Witt-condition}
\forall \, y \in \R^b: \ \spec S(y) \cap \left[-1, 1\right] = \varnothing,
\end{equation}

\begin{remark}
We should point out that Albin and Gell-Redman \cite{Albin-Jesse} require 
a smaller spectral gap $\spec S(y) \cap \Bl-1/2, 1/2\Br = \varnothing$. However, 
when proving an analogue of the crucial \cite[Lemma 3.10]{Albin-Jesse} by explicit
computations, it seems that a smaller spectral gap may not be sufficient for our purposes. 
In any case, if $D$ is the Gauss--Bonnet operator on a 
stratified Witt space, one can always achieve 
$\spec S(y) \cap \Bl-R, R\Br = \varnothing$ for any $R>0$ by a simple rescaling of the metric.
\end{remark}

\subsection{Squares of generalized Dirac operators}

In view of the commutator relations \eqref{Eq1},
the generalized Laplace operators $D^2$ and $D_{y_0}^2$, acting both 
on $C_0^\infty(\R_+\times\R^b,H^\infty)$, are of the following form
\begin{equation*}
\begin{split}
&D^2 = -\b_x^2 + X^{-2}\ S \ (S+1) + T^2, \\
&D^2_{y_0} = -\b_x^2 + X^{-2} \ S(y_0)  \ (S(y_0) + 1) + T^2_{y_0}.
\end{split}
\end{equation*}
We set $A := \left|S\right| +\frac{1}{2}$. Assuming 
$\spec S \cap \left[-1, 1\right] = \varnothing$, we find $S(S+1) = A^2 -1/4$
and rewrite the generalized Laplacians $D^2$ and $D_{y_0}^2$ as follows
\begin{equation*}
\begin{split}
&D^2 = -\b_x^2 + X^{-2} \Bl A^2 - \frac{1}{4} \Br + T^2, \\
&D^2_{y_0} = -\b_x^2 + X^{-2}\Bl A^2(y_0) - \frac{1}{4} \Br+ T^2_{y_0}.
\end{split}
\end{equation*}
As before, we may apply the Fourier transform $\mathfrak{F}_{y\to\xi}$ 
and compute
\begin{equation*}
\begin{split}
\mathfrak{F}_{y\to\xi} \circ D^2_{y_0} \circ \mathfrak{F}_{y\to\xi}^{-1} &= -
\b_x^2 + X^{-2}\left(A^2(y_0) - \frac{1}{4}\right) + c(\xi,y_0)^2 \\ 
&=: L^2(y_0,\xi), \ \textup{where} \ 
c(\xi; y_0)^2 =- \sum_{j,k=1}^b c_j(y_0) c_k(y_0)\xi_j \xi_k.
\end{split}
\end{equation*}

\subsection{Sobolev-spaces of an abstract edge}

Recall the definition of interpolation scales of Hilbert spaces in \S \ref{s.ISHS}. This defines for 
each $y_0\in\R^b$ an interpolation scale $H^s(S(y_0))$, $s\in\R$.
We can now define the Sobolev-scales on the 
{\it model cone} and the {\it model edge} in our abstract setting. Consider for 
this the Sobolev-scale $H^\bullet_e(\R_+)$ generated by\footnote{The edge
Sobolev scale $H^\bullet_e(\R_+)$ prescribes regularity under differentiation by $x\partial_x$.
However, $x\partial_x$ is not a symmetric operator and hence we take its symmetrization $(i x\partial_x + i/ 2)$
as the generator of the Sobolev scale. Alternatively we can replace the definition of Sobolev scales to 
allow for closed not necessarily symmetric operators.} $(i x\partial_x + i/ 2)$; and 
the Sobolev-scale \LHchange{There is an arrow in the note that I do not 
understand.} $H_e^\bullet(\R_+\times \R^b)$ generated by 
$\gL = (i x\partial_x + i/ 2)\hot I+I\hot x T_{y_0}$. The lower index $e$ indicates
that these interpolation scales coincide with the edge Sobolev spaces for integer
orders.

\begin{definition} Let $y_0 \in \R^b$ be fixed.
	\begin{enumerate}
		\item[a)] The Sobolev-scale $W^\bullet(\R_+,H)$ of an abstract model 
		cone is defined as an interpolation scale with generator 
		$(i x\partial_x + i/ 2)\hot I+I\hot S(y_0)$. By Proposition \ref{p.ISHS.12}
		\begin{equation}\label{Sobolev-space-definition}
		W^s(\R_+,H):=(H^s_e(\R^+)\hot H) \cap 
		(L^2(\R_+)\hot H^s(S(y_0))).
		\end{equation}
	
		\item[b)] The Sobolev-scale $W^\bullet(\R_+\times \R^b, H)$ of an 
		abstract model edge is defined as an interpolation scale with 
		generator $\gL\hot  I + I \hot S(y_0)$, where $\gL$ is the generator of the
		Sobolev-scale $H^s_e(\R_+\times \R^b)$. By Proposition \ref{p.ISHS.12}
		\begin{equation}
		W^s(\R_+\times\R^b,H):=(H^s_e(\R_+\times\R^b)\hot H) \cap 
		(L^2(\R_+\times \R^b)\hot H^s(S(y_0))).
		\end{equation}
	\end{enumerate}
\end{definition}

\begin{remark}\label{different-sobolev-scales}
In view of Proposition \ref{p.ISHS.7}, for $y,y_0\in \R^b$, the interpolation scales of $S(y)$ 
and $S(y_0)$ need not coincide. However, since for any $y\in \R^b$,
the domain of $S(y)$ is fixed and given by $\sD_S$, we have $H^s(S(y))= H^s(S(y_0))$ for 
$0\leq s\leq 1$. In particular the Sobolev scales $W^s(\R_+, H)$ and $W^s(\R_+\times \R^b, H)$
do not depend on $y_0 \in \R^b$ for $0\leq s\leq 1$. In fact, in our 
arguments below we will require
independence of the Sobolev spaces for $0 \leq s \leq 2$. 
\end{remark}

We conclude with a definition of weighted Sobolev-spaces, where we denote by $X$ 
the multiplication operator by $x \in \R_+$.

\begin{definition}
	The weighted Sobolev-scales are defined by
	\begin{equation}
	W^{s,\delta,l}:= X^\delta (1+X)^{-l} W^s(\R_+,H), \quad W^{s,\delta}:= W^{s,\delta,0}.
	\end{equation}
\end{definition}

\subsection{Examples of generalized Dirac operators on an abstract edge}
\label{s.AEDSE}

The spin Dirac operator on a model edge space
is indeed a generalized Dirac operator in the sense that it is given by the differential 
expression \eqref{generalized-Dirac} and satisfies the commutator relations \eqref{Eq1}. This has been 
established by Albin and Gell-Redman \cite{Albin-Jesse}. In this subsection we 
prove that the Gauss--Bonnet operator on a model edge space is a generalized 
Dirac operator in the sense above as well.

Let $M^m$ and $N^n$ be Riemannian manifolds. Given forms $\omega_p \in \Omega^p(M)$ and $\eta_q \in \Omega^q(N)$, we will write $\omega_p \wedge \eta_q$ for the form $\pi_M^* (\omega_p) \wedge \pi_N^* (\eta_q) \in \Omega^{p+q}(M\times N)$, where $\pi_M:M\times N \to M$ and $\pi_N:M\times N \to N$ are projections onto the first and second factors respectively. It is well known that the 
exterior derivative $d:\Omega^{\ast}(M\times N) \to \Omega^{\ast}(M\times N)$ 
satisfies the Leibniz rule, \ie if $\omega_p\in\Omega^{p}(M)$ and 
$\eta_q\in\Omega^q(N)$ then
\begin{equation}
d(\omega_p\wedge\eta_q) \ = \ (d^M\omega_p)\wedge\eta_q + 
(-1)^{p} \  
\omega_p\wedge(d^N\eta_q).
\end{equation}

\begin{lemma}\label{Leibniz}
	The same Leibniz rule holds for the adjoint of the exterior derivative 
	$d^t$ in $\Omega^{\ast}(M\times N)$.
\end{lemma}

\begin{proof}
Note that $\Omega^{\ast}(M\times N)$ can be decomposed into a direct sum of 
subspaces of the form $\Omega^{\ast}(M) \wedge \Omega^{\ast}(N)$. Hence it suffices
to study the action of $d^t$ on differential forms in $\Omega^{p+1}(M) \wedge \Omega^q(N)$,
where we have
\begin{equation}
d^t: \Omega^{p+1}(M) \wedge \Omega^q(N) \to 
(\Omega^{p}(M) \wedge \Omega^q(N)) \oplus (\Omega^{p}(M) \wedge \Omega^{q-1}(N)).
\end{equation}
	Consider $\tilde\omega_{p}\in \Omega^{p}(M)$, 
	$\omega_{p+1},\tilde\omega_{p+1}\in 
	\Omega^{p+1}(M)$, $\tilde\eta_q, \eta_q\in \Omega^{q}(N)$ and 
	$\tilde\eta_{q-1}\in \Omega^{q-1}(N)$, then we have for the first component of $d^t$
	\begin{equation}
	\begin{aligned}
	\langle d_{p+q}^t (\omega_{p+1}\wedge&\eta_q), 
	\tilde\omega_{p}\wedge\tilde\eta_q\rangle \\
	&= \langle 
	\omega_{p+1}\wedge\eta_q, 
	(d^M\tilde\omega_{p})\wedge\tilde\eta_q + (-1)^p 
	\tilde\omega_{p}\wedge(d^N\tilde\eta_q)\rangle\\
	&=\langle 
	\omega_{p+1}\wedge\eta_q, 
	(d^M\tilde\omega_{p})\wedge\tilde\eta_q \rangle\\
	&=\langle \omega_{p+1},  (d^M\tilde\omega_{p})\rangle_M
	\langle \eta_{q}, \tilde\eta_q\rangle_N\\
	&=\langle (d^{M,t}\omega_{p+1}), \tilde\omega_{p}\rangle_M
	\langle \eta_{q}, \tilde\eta_q\rangle_N\\
	&=\langle (d^{M,t}\omega_{p+1})\wedge \eta_{q}, 
	\tilde\omega_{p}\wedge\tilde\eta_{q}\rangle.
	\end{aligned}
	\end{equation}
	For the second component of $d^t$, we obtain
	\begin{equation}
	\begin{aligned}
	\langle d_{p+q}^t &(\omega_{p+1}\wedge\eta_q), 
	\tilde\omega_{p+1}\wedge\tilde\eta_{q-1}\rangle\\ 
	&= \langle 
	\omega_{p+1}\wedge\eta_q, 
	(d^M\tilde\omega_{p+1})\wedge\tilde\eta_{q-1} + (-1)^{p+1} 
	\tilde\omega_{p+1}\wedge(d^N\tilde\eta_{q-1})\rangle\\
	&=(-1)^{p+1}\langle 
	\omega_{p+1}\wedge\eta_q, 
	\tilde\omega_{p+1}\wedge(d^N\tilde\eta_{q-1}) \rangle\\
	&=(-1)^{p+1}\langle \omega_{p+1},  \tilde\omega_{p+1}\rangle_M
	\langle \eta_{q},  (d^N\tilde\eta_{q-1})\rangle_N\\
	&=(-1)^{p+1}\langle \omega_{p+1}, \tilde\omega_{p+1}\rangle_M
	\langle (d^{N,t}\eta_{q}), \tilde\eta_{q-1}\rangle_N\\
	&=(-1)^{p+1}\langle \omega_{p+1}\wedge (d^{N,t}\eta_{q}), 
	\tilde\omega_{p+1}\wedge\tilde\eta_{q-1}\rangle.
	\end{aligned}
	\end{equation}
	Altogether, we arrive at the result
	\begin{equation}
	d_{p+q}^t (\omega_{p+1}\wedge\eta_q) = (d_p^{M,t}\omega_{p+1})\wedge 
	\eta_{q} + (-1)^{p+1} \omega_{p+1}\wedge (d^{N,t}_{q} \eta_q).
	\end{equation}
	
\end{proof}

We now apply Lemma \ref{Leibniz} to the case of a model edge $C(F) \times Y$
of cones $C(F) = \R_+\times F$ fibered over an edge manifold $Y$.
Recall that on a cone $C(F) = \R_+\times F$ we have as in 
\cite[(5.9a), (5.9b)]{BS2} the following isometric identifications 
\begin{equation}\label{unitary-trafo}
\Omega^{\rm ev}(C(F)) \cong C^{\infty}(\R_+,\Omega^{\ast}(F)), \qquad
\Omega^{\rm odd}(C(F)) \cong C^{\infty}(\R_+,\Omega^{\ast}(F)).
\end{equation}
Under these identifications the Gauss--Bonnet 
operator $D=d+d^t$ acting now from $\Omega^{\rm ev}(C(F))\cong C^{\infty}(\R_+,\Omega^{\ast}(F))$ to $\Omega^{\rm odd}(C(F))\cong C^{\infty}(\R_+,\Omega^{\ast}(F))$, takes the form
cf. \cite[(5.10)]{BS2} 
\begin{equation}
D = \frac{d}{dx}+X^{-1} A.
\end{equation}
Respectively, the full operator $D$ acts on 
$C^{\infty}(\R_+,\Omega^{\ast}(C(F))\oplus \Omega^{\ast}(C(F)))$ as
\begin{equation}
\left(
\begin{array}{cc}
0 & -\frac{d}{dx}+X^{-1} A\\
\frac{d}{dx}+X^{-1} A & 0
\end{array}
\right) = 
\left(
\begin{array}{cc}
0 & -1\\
1 & 0
\end{array}
\right) 
\left(\frac{d}{dx} + X^{-1} 
\left(
\begin{array}{cc}
A & 0\\
0 & -A
\end{array}\right)\right).
\end{equation}
Note that the grading operator on $\Omega^{\ast}(C(F))\oplus \Omega^{\ast}(C(F))$ is 
$\left(\begin{array}{cc}
1 & 0\\
0 & -1
\end{array}\right)$.

Taking now the cartesian product by a manifold $Y$ (the edge), we have
\begin{equation*}
\begin{aligned}
\Omega^{\rm ev}(C(F)\times Y) &=  \Omega^{\rm ev}(C(F))\otimes \Omega^{\rm 
ev}(Y) \oplus \Omega^{\rm odd}(C(F))\otimes \Omega^{\rm odd}(Y)\\
&\cong C^{\infty}(\R_+, \Omega^{\ast}(F))\otimes \Omega^{\rm ev}(Y) \oplus 
C^{\infty}(\R_+, \Omega^{\ast}(F))\otimes \Omega^{\rm odd}(Y),
\end{aligned}
\end{equation*}
where we used the identifications \eqref{unitary-trafo} in the second equality. 
In exactly the same manner we find for differential forms of odd degree
\begin{equation*}
\begin{aligned}
\Omega^{\rm odd}(C(F)\times Y) &=  \Omega^{\rm odd}(C(F))\otimes \Omega^{\rm 
	ev}(Y) \oplus \Omega^{\rm ev}(C(F))\otimes \Omega^{\rm odd}(Y)\\
&\cong C^{\infty}(\R_+, \Omega^{\ast}(L))\otimes \Omega^{\rm ev}(Y) \oplus 
C^{\infty}(\R_+, \Omega^{\ast}(F))\otimes \Omega^{\rm odd}(Y).
\end{aligned}
\end{equation*}
So again we have an identification of the space $\Omega^{\rm 
ev}(C(F)\times Y)$ with the space $\Omega^{\rm odd}(C(F)\times Y)$. 
For $\omega_1 \in \Omega^{\rm ev}(C(F))$, $\omega_2 \in \Omega^{\rm odd}(C(F))$, $\eta_1 \in 
\Omega^{\rm ev}(Y)$ and $\eta_2 \in \Omega^{\rm odd}(Y)$, we have 
$\omega_1\otimes \eta_1 \oplus \omega_2\otimes\eta_2\in \Omega^{\rm 
ev}(C(F)\times Y)$. Using Lemma \ref{Leibniz} we now find for $D= d + d^t$,
\begin{equation}
\begin{aligned}
D(\omega_1\otimes \eta_1 \oplus \omega_2\otimes\eta_2) &= 
D^{C(F)}\omega_1\otimes \eta_1 + \omega_1 \otimes D^Y\eta_1\\
&+ D^{C(F)}\omega_2\otimes \eta_2 - \omega_2\otimes D^Y\eta_2\\
&=\left(
\begin{array}{cc}
\partial_x + X^{-1}A & - D^Y\\
D^Y & -\partial_x + X^{-1}A
\end{array}
\right) 
\left(
\begin{array}{c}
\omega_1\otimes\eta_1\\
\omega_2\otimes\eta_2
\end{array}
\right).
\end{aligned}
\end{equation}
Note that by construction $A$ and $D^Y$ commute. By abuse of notation $A$ 
acts as $A\otimes I$ and $D^Y$ acts as $I\otimes D^Y$ on the tensors. The 
full Gauss--Bonnet then becomes
\begin{equation}\label{GaussBonnetOp}
D= \left(\begin{array}{cccc}
0 & 0 & -\partial_x + X^{-1}A & D^Y\\
0 & 0 & -D^Y & \partial_x + X^{-1}A\\
\partial_x + X^{-1}A & -D^Y & 0 & 0\\
D^Y & -\partial_x + X^{-1}A & 0 & 0
\end{array}\right).
\end{equation}
This expression can rewritten as follows.
\begin{equation}
\begin{aligned}
D = \left(
\begin{array}{cccc}
0 & 0 & -1 & 0\\
0 & 0 &  0 & 1\\
1 & 0 &  0 & 0\\
0 &-1 &  0 & 0
\end{array}
\right) &\left(\partial_x + X^{-1} \left(
\begin{array}{cccc}
1 & 0 & 0 & 0\\
0 & -1 &0 & 0\\
0 & 0 & -1&0\\
0 & 0 & 0 & 1
\end{array}
\right)\cdot A\right)\\
&+ 
\left(
\begin{array}{cccc}
0 & 0 & 0 & 1\\
0 & 0 & -1 & 0\\
0 & -1 & 0 & 0\\
1 & 0 & 0 & 0
\end{array}
\right)\cdot D^Y,
\end{aligned}
\end{equation}
with grading operator $\left(\begin{array}{cc}
I_2 & 0\\
0 & - I_2
\end{array}\right)$
where $I_2$ is the identity in $M_2(\R)$. Define the following matrices
\begin{equation*}
\Gamma = \left(
\begin{array}{cccc}
0 & 0 & -1 & 0\\
0 & 0 &  0 & 1\\
1 & 0 &  0 & 0\\
0 &-1 &  0 & 0
\end{array}
\right)\!\!,
S =\left(
\begin{array}{cccc}
1 & 0 & 0 & 0\\
0 & -1 &0 & 0\\
0 & 0 & -1&0\\
0 & 0 & 0 & 1
\end{array}
\right) A,
T = \left(
\begin{array}{cccc}
0 & 0 & 0 & 1\\
0 & 0 & -1 & 0\\
0 & -1 & 0 & 0\\
1 & 0 & 0 & 0
\end{array}
\right) D^Y\!.
\end{equation*}
We introduce the usual Clifford matrices
\begin{equation}
\sigma_1 = \left(
\begin{array}{cc}
0 & -1\\
1 & 0
\end{array}
\right),\
\sigma_2=\left(
\begin{array}{cc}
0 & i \\
i & 0
\end{array}
\right),\
\omega=
\left(
\begin{array}{cc}
1 & 0\\
0 & -1
\end{array}
\right) = i \cdot \sigma_1 \cdot \sigma_2.
\end{equation}
We have,
\begin{equation}
\begin{aligned}
\Gamma &= \left(
\begin{array}{cc}
0 & -\omega\\
\omega & 0
\end{array}
\right) = \sigma_1 \otimes \omega,\\
S &= \left(
\begin{array}{cc}
\omega & 0\\
0 & -\omega
\end{array}
\right)\otimes A = \omega \otimes \omega \otimes A,\\
T &= \left(
\begin{array}{cc}
0 & -\sigma_1\\
\sigma_1 & 0
\end{array}
\right) \otimes D^Y = \sigma_1 \otimes \sigma_1 \otimes D^Y.
\end{aligned}
\end{equation}
We can now easily compute the commutator relations
\begin{equation}\label{commutator-GB}
\begin{aligned}
\Gamma \ S + S \ \Gamma &= \sigma_1 \otimes \omega \cdot \omega \otimes 
\omega \otimes A + \omega \otimes \omega \otimes A \cdot \sigma_1 \otimes 
\omega\\
&= (\sigma_1 \omega + \omega \sigma_1)\otimes \omega^2 \otimes A  = 0.\\
\Gamma \ T + T \ \Gamma &= \sigma_1 \otimes \omega \cdot \sigma_1 \otimes 
\sigma_1 \otimes D^Y + \sigma_1 \otimes \sigma_1 \otimes D^Y \cdot \sigma_1 
\otimes 
\omega\\
&=  \sigma_1 \otimes (\omega \sigma_1 + \sigma_1 \omega)\otimes D^Y = 0.\\
T \ S - S \ T &= \sigma_1 \otimes \sigma_1 \otimes D^Y \cdot \omega \otimes 
\omega \otimes A - \omega \otimes \omega \otimes A \cdot \sigma_1 \otimes 
\sigma_1 \otimes D^Y\\
&= (\sigma_1\cdot \omega 
\otimes \sigma_1\cdot \omega - \omega \cdot \sigma_1 \otimes \omega \cdot \sigma_1)\otimes D^Y \cdot A\\
&=(\omega \cdot \sigma_1 \otimes \sigma_1 \cdot \omega+ \sigma_1\cdot \omega 
\otimes  \sigma_1 \cdot \omega)\otimes D^Y \cdot A\\
&=(\omega \cdot \sigma_1 + \sigma_1\cdot \omega)\otimes  \sigma_1 \cdot \omega \otimes D^Y \cdot A = 0.
\end{aligned}
\end{equation}

\section{Some integral operators and auxiliary estimates}

In this section we study boundedness properties of certain 
integral operators that appear below when inverting the model 
Bessel operator $L^2(y_0,\xi)$ and its square $L^2(y_0,\xi)^2$. 

\begin{proposition}\label{Schur-1-1}
	Let $\nu \geq \frac{3}{2} + \delta$ for some $\delta>0$ and consider the integral operator $K$
	acting on $C^\infty_0(\R_+)$ with integral 
	kernel given by 
	\begin{equation}
	k(x,y) = \left\{\begin{array}{cc}
	\frac{1}{2\nu} \bl \frac{y}{x}\br^{\nu} (xy)^{\frac{1}{2}}, & y\leq 
	x,\\
	\frac{1}{2\nu} \bl \frac{y}{x}\br^{-\nu} (xy)^{\frac{1}{2}}, & x\leq 
	y.
	\end{array}	
	\right.
	\end{equation}
	Then $X^{-2} \circ K$ defines a bounded operator on $L^2(0,\infty)$	
	and there exists a constant $C>0$ depending only on $\delta>0$ 
	such that 
	\begin{equation}\label{K-0-2}
	\begin{split}
	&\|X^{-2}  \circ K \|_{L^2\to L^2}\leq 
	\Bl\nu^{2}-\frac{9}{4}\Br^{-1}, \\
	&\|(X \partial_x) \circ X^{-2} \circ K \|_{L^2\to L^2}\leq 
	\Bl\nu-\frac{3}{2}\Br^{-1}, \\
	&\|(X \partial_x)^2 \circ X^{-2} \circ K \|_{L^2\to L^2}\leq C.
	\end{split}
	\end{equation}
	\end{proposition}

\begin{proof}
	We apply Schur's test, cf. Halmos and Sunder 
	\cite[Theorem 5.2]{HaSu78}. We have
	\begin{equation}
	\begin{aligned}
	\int_{0}^{x}x^{-2}k(x,y) \ dy &+ \int_{x}^{\infty} x^{-2} k(x,y) \ dy \\ &= 
	\frac{1}{2\nu} \Bl x^{-\nu - \frac{3}{2}}\int_{0}^x y^{\nu+\frac{1}{2}} \ 
	dy + x^{\nu - \frac{3}{2}} \int_{x}^{\infty} y^{-\nu + \frac{1}{2}} 
	dy\Br\\
	&= \frac{1}{2\nu} \Bl \frac{1}{\nu + \frac{3}{2}} + \frac{1}{\nu - 
	\frac{3}{2}} \Br =  \Bl \nu^{2} - \frac{9}{4} \Br^{-1}.
	\end{aligned}
	\end{equation}
	Similarly, we integrate in the $x$ variable and find
	\begin{equation}
	\begin{aligned}
	\int_{0}^{y}x^{-2}k(x,y) \ dx &+ \int_{y}^{\infty} x^{-2} k(x,y) \ dx \\&= 
	\frac{1}{2\nu} \Bl y^{-\nu + \frac{1}{2}}\int_{0}^y x^{\nu-\frac{3}{2}} \ 
	dx + y^{\nu + \frac{1}{2}} \int_{y}^{\infty} x^{-\nu - \frac{3}{2}} 
	dx\Br\\
	&= \frac{1}{2\nu} \Bl \frac{1}{\nu - \frac{1}{2}} + \frac{1}{\nu + 
		\frac{1}{2}} \Br =  \Bl \nu^{2} - \frac{1}{4} \Br^{-1}
	\end{aligned}
	\end{equation}
	From there one concludes that 
	\begin{equation}
	\|X^{-2} \circ K\|_{L^2\to L^2}\leq 
	\left( \Bl \nu^2 - \frac{9}{4}\Br \Bl \nu^{2}-\frac{1}{4}\Br\right)^{-\frac{1}{2}} \leq 
	\Bl\nu^2-\frac{9}{4}\Br^{-1}.
	\end{equation} 
	This proves the first estimate. The second and third estimates are established
	ad verbatim.
\end{proof}

\begin{proposition}\label{Schur-2-1}
	Let $\nu \geq \frac{3}{2} + \delta$ for some $\delta>0$ and 
	let $\beta>0$ be positive real number. Consider integral 
	operator $K$ acting on $C^\infty_0(\R_+)$ with integral 
	kernel given in terms of modified Bessel 
	functions by 
	\begin{equation}
	k(x,y) = \left\{ \begin{array}{cc}
	(xy)^{\frac{1}{2}} I_\nu (\beta \ y) K_{\nu}(\beta \ x), & y\leq x,\\
	(xy)^{\frac{1}{2}} I_\nu (\beta \ x) K_{\nu}(\beta \ y), & x\leq y.
	\end{array}
	\right.
	\end{equation}
	Then $X^{-2} \circ K$ defines a bounded operator on $L^2(0,\infty)$	
	and there exists a constant $C>0$ depending only on $\delta>0$ 
	such that 
	\begin{equation}\label{K2-statement}
	\begin{split}
	&\|X^{-2} \circ K \|_{L^2\to L^2}\leq 
	C \Bl\nu^{2}-\frac{9}{4}\Br^{-1}, \\
	&\|(X \partial_x) \circ X^{-2} \circ K \|_{L^2\to L^2}\leq 
	C \Bl\nu-\frac{3}{2}\Br^{-1}, \\
	&\|(X \partial_x)^2 \circ X^{-2} \circ K \|_{L^2\to L^2}\leq C.
	\end{split}
	\end{equation}
	\end{proposition}

\begin{proof}
Following Olver \cite[p. 377 (7.16), (7.17)]{Olv:AAS}, 
	we note the asymptotic expansions for Bessel functions
	as $\nu \to \infty$
	\begin{equation}\label{EqBesselestimates}
	I_\nu (\nu x)  \sim \frac{1}{\sqrt{2\pi\nu}}\cdot 
	\frac{e^{\nu \cdot\eta(x)}}{(1+x^2)^{1/4}},\;\; 
	K_\nu (\nu x)  \sim \sqrt{\frac{2\pi}{\nu}}\cdot 
	\frac{e^{-\nu \cdot\eta(x)}}{(1+x^2)^{1/4}}
	\end{equation}
	where $\eta(x) = \sqrt{1+x^2} +\ln \frac{x}{1+\sqrt{1+x^2}}$. By 
	\eqref{variation-bounds-2}, these expansions are uniform in 
	$x\in(0,\infty)$. We define an auxiliary function
	\begin{equation}
	E(x, \nu):=\frac{e^{\nu(\eta(x)-\ln x)}}{(1+x^2)^{\frac{1}{4}}}.	
	\end{equation}
	Note that $\eta(x) - \ln x$ is 
	increasing as $x\to\infty$, since
	\begin{equation}\label{eta-grow}
	\begin{aligned}
	(\eta(x)-\ln x)' &= (\sqrt{1+x^2} - \ln(1+\sqrt{1+x^2})'\\
	&=2x\Bl\frac{1}{2\sqrt{1+x^2}} - \frac{1}{2\sqrt{1+x^2}}\cdot 
	\frac{1}{1+\sqrt{1+x^2}} \Br\\
	&=\frac{x}{\sqrt{1+x^2}}\cdot \frac{\sqrt{1+x^2}}{1+\sqrt{1+x^2}}>0.
	\end{aligned}
	\end{equation}
	Consequently $E(x,-\nu)$ is decreasing and for $y \leq x$
	\begin{equation}\label{EquationKIestimate}
	\begin{aligned}
	\left|K_{\nu+\alpha}(x) \cdot I_\nu(y)\right| &\leq C\cdot 
	\frac{1}{\sqrt{\nu(\nu+\alpha))}} \Bl\frac{y}{\nu}\Br^\nu 
	 \Bl\frac{x}{\nu+\alpha}\Br^{-(\nu+\alpha)} \\ &\times E\left(\frac{y}{\nu},\nu\right)
	\cdot E\left(\frac{y}{\nu+\alpha},-(\nu+\alpha)\right),
	\end{aligned}
	\end{equation}
	for some uniform constant $C>0$ and $\alpha \in \{0,1\}$. In fact, below we will always denote 
	uniform positive constants by $C$. We proceed with a technical calculation
	\begin{equation}
	\begin{aligned}
	&(\nu+\alpha)\Bl\eta\Bl\frac{y}{\nu+\alpha}\Br - \ln 
		\frac{y}{\nu+\alpha} \Br - 
		\nu\Bl\eta\Bl\frac{y}{\nu}\Br - 
		\ln \frac{y}{\nu} \Br \\
	&=\sqrt{(\nu+\alpha)^2 +y^2} - \sqrt{\nu^2 +y^2} + 
		\nu\ln\frac{\nu+\alpha}{\nu} \\
	&- \nu \ln \Bl \frac{\nu+\alpha +  
		\sqrt{(\nu+\alpha)^2+y^2}}{\nu+\sqrt{\nu^2+y^2}} 
		\Br-\alpha\ln \Bl 
		1+\sqrt{1+\Bl \frac{y}{\nu+\alpha}\Br^2} \Br\\
	&=\sqrt{\nu^2 +y^2} 
	\Bl\sqrt{1+\frac{2\alpha\nu+\alpha^2}{\nu^2+y^2}}-1 
		\Br + \nu \ln\Bl 1+\frac{\alpha}{\nu}\Br\\
	&-\nu\ln \Bl 1+ \frac{\alpha}{\nu+\sqrt{\nu^2+y^2}} 
		+\frac{\sqrt{\nu^2+y^2}}{\nu+\sqrt{\nu^2+y^2}}\Bl 
		\sqrt{1+\frac{2\alpha\nu+\alpha^2}{\nu^2+y^2}}-1\Br \Br\\
	&-\alpha \ln \Bl 1 + \sqrt{1+\Bl \frac{y}{\nu+\alpha}\Br^2}\Br.
	\end{aligned}
	\end{equation}
	In order to continue with our estimates we write $O(f)$ for any function 
	whose absolute value is bounded by $f$, with a uniform constant that is 
	independent of $\nu$ and $y$, and note
	\begin{enumerate}
	\item $1+\frac{2\alpha\nu+\alpha^2}{\nu^2+y^2}$ is always positive,
	\item $\ln\frac{\nu+\alpha}{\nu} = \frac{\alpha}{\nu} + O\Bl \frac{1}{\nu^2}\Br$
	where the $O$-constant depends on $\delta>0$,
	\item 
	\begin{align*}
	\ln \Bl 1+ \frac{\alpha}{\nu+\sqrt{\nu^2+y^2}} 
		+\frac{\sqrt{\nu^2+y^2}}{\nu+\sqrt{\nu^2+y^2}}
		\Bl \sqrt{1+\frac{2\alpha\nu+\alpha^2}{\nu^2+y^2}}-1\Br \Br
		\\ = \frac{\alpha}{\nu+\sqrt{\nu^2+y^2}} 
		+\frac{\sqrt{\nu^2+y^2}}{\nu+\sqrt{\nu^2+y^2}} + O\Bl \frac{1}{\nu^2}\Br,
		\end{align*}
		where the $O$-constant may be chosen independently of $y\in (0,\infty)$, but
		depends on $\delta>0$.
	\end{enumerate}
	Plugging in these observations, we arrive at the following estimate,
	\begin{equation}
	\begin{aligned}
	&(\nu+\alpha)\Bl\eta\Bl\frac{y}{\nu+\alpha}\Br - \ln 
		\frac{y}{\nu+\alpha} \Br - 
		\nu\Bl\eta\Bl\frac{y}{\nu}\Br - 
		\ln \frac{y}{\nu} \Br \\
	&=-\alpha\ln\Bl 1+\sqrt{1+\Bl \frac{y}{\nu}\Br^2} \Br + 
	O\bl 1\br.
	\end{aligned}
	\end{equation}
     Plugging this into the 
	estimate \eqref{EquationKIestimate} we obtain:
	\begin{equation*}
	\begin{aligned}
	\left|K_{\nu+\alpha}(x) \cdot I_\nu(y)\right| &\leq C \cdot 
	\frac{1}{\sqrt{\nu(\nu+\alpha))}} \cdot F\Bl \frac{y}{\nu}\Br \cdot 
	 \Bl \frac{y}{x}\Br^\nu \cdot\Bl\frac{\nu+\alpha}{\nu} \Br^\nu \cdot
	\Bl\frac{x}{\nu+\alpha} \Br^{-\alpha} \\ &\leq C \cdot 
	\frac{1}{\sqrt{\nu(\nu+\alpha))}} \cdot F\Bl \frac{y}{\nu}\Br \cdot 
	 \Bl \frac{y}{x}\Br^\nu \cdot \Bl\frac{x}{\nu+\alpha} \Br^{-\alpha}, \\
	 & \textup{where} \ F\Bl \frac{y}{\nu}\Br := \Bl 1+\sqrt{1+\Bl \frac{y}{\nu}\Br^2} 
	\Br^{\alpha} / \sqrt{1+\Bl \frac{y}{\nu}\Br^2},
	\end{aligned}
	\end{equation*}
	for some uniform constants $C>0$, depending only on $\delta$, where in the second inequality 
	we noted that	$\lim\limits_{\nu\to\infty} \Bl \frac{\nu+\alpha}{\nu} \Br^\nu = 
	e^\alpha$ and hence $\Bl \frac{\nu+\alpha}{\nu} \Br^\nu$ is bounded 
	uniformly for large $\nu$. We also note that $(\nu (\nu+\alpha))^{-1} \leq C \nu^{-2}$, as long as 
	$\nu$ and $(\nu+\alpha)$ are positive bounded away from zero. Hence we arrive
	at the following estimate
	\begin{equation*}
	\begin{aligned}
	\left|K_{\nu+\alpha}(x) \cdot I_\nu(y)\right| \leq C \cdot 
	\frac{1}{\nu} \cdot F\Bl \frac{y}{\nu}\Br \cdot 
	 \Bl \frac{y}{x}\Br^\nu \cdot \Bl\frac{x}{\nu+\alpha} \Br^{-\alpha}.	
	 \end{aligned}
	\end{equation*}
Note that for $\alpha = 1$, $F(y/\nu)$ is uniformly bounded
and for $\alpha=0$, $F(y/\nu)\leq C(y/\nu)^{-1}$. Hence we conclude
for $x \geq y$ and some uniform constant $C>0$
		\begin{equation}\label{combinations}
		\begin{split}
	&\left| \, K_{\nu}(\beta \ x) \cdot I_\nu(\beta \ y) \, \right| 
	\leq C \cdot \frac{1}{\nu} \cdot \Bl\frac{y}{x}\Br^{\nu}. \\
	&\left| \, x K_{\nu+1}(\beta \ x) \cdot I_\nu(\beta \ y) \, \right| 
	\leq C \cdot \Bl\frac{y}{x}\Br^{\nu}. \\
	&\left| \, y K_{\nu}(\beta \ x) \cdot I_{\nu - 1}(\beta \ y) \, \right| 
	\leq C \cdot \Bl\frac{y}{x}\Br^{\nu}.	
	 \end{split}
	\end{equation}
By the formulae for the derivatives of modified Bessel functions
	\begin{equation*}
		\begin{split}
	& (x\partial_x) I_\nu(x) = xI_{\nu-1}(x) - \nu I_\nu(x), \\
	& (x\partial_x) K_\nu(x) = \nu K_\nu(x) - xK_{\nu+1}(x),
	 \end{split}
	\end{equation*}
the derivatives $(x\partial_x) k(x,y)$ and $(x\partial_x)^2 k(x,y)$ can be written as
combinations of the products in \eqref{combinations}. 
In view of Proposition \ref{Schur-1-1}, we obtain the result.
\end{proof}

\begin{remark}
The statement of Proposition \ref{Schur-2-1} corresponds to Br\"uning-Seeley 
\cite[Lemma 3.2]{BruSee:TRE}. However, the latter reference does not provide 
an exact lower bound on $\nu$, which is crucial in order to establish the optimal
spectral gap in the spectral Witt condition.
\end{remark}

We close the section with a crucial observation.

\begin{corollary}\label{comp-asymptotics-infinity}
Let $\nu \geq \frac{3}{2} + \delta$ for some $\delta>0$ and let
$K$ denote either the integral operator in Proposition \ref{Schur-1-1}
or in Proposition \ref{Schur-2-1}. Then for any $u \in L^2(\R_+)$ with compact
support in $[0,1]$, $Ku$ admits the following estimates \footnote{$Ku$ is 
continuously differentiable on $(0,\infty)$.}
\begin{align}
| Ku(x) | \leq \frac{C}{\nu} \|u\|_{L^2} \ x^{-1-\delta}, \quad | (x 
\partial_x) Ku(x) |
\leq C \|u\|_{L^2} \ x^{-1-\delta},
\end{align}
for a constant $C>0$ independent of $u$ and $\nu$.
\end{corollary}
\begin{proof}
It suffices to prove the statement for $K$ in Proposition \ref{Schur-1-1},
since by \eqref{combinations} the integral kernels in Proposition \ref{Schur-2-1},
and their derivatives, can be estimated 
against those in Proposition \ref{Schur-1-1}. Consider $u \in L^2(\R_+)$ 
such that $\supp u \subset [0,1]$. Then for $x>1$ we find using $\nu \geq \frac{3}{2} + \delta$
\begin{align*}
| Ku(x) | &\leq \frac{x^{-\nu+\frac{1}{2}}}{2\nu} \int_0^1 y^{\nu+\frac{1}{2}} | u(y) | dy 
\leq \frac{C}{\nu} \|u\|_{L^2}\ x^{-1-\delta}, \\
| x\partial_x Ku(x) | &\leq \frac{(-\nu+\frac{1}{2})}{2\nu} x^{-\nu+\frac{1}{2}} \int_0^1 y^{\nu+\frac{1}{2}} |u(y)| dy
\leq C \|u\|_{L^2}\ x^{-1-\delta}, \\
\end{align*}
for a constant $C>0$ independent of $u$ and $\nu$.
\end{proof}

\section{Invertibility of the model Bessel operators}

In this section we prove invertibility of  
\begin{equation}
L(y_0,\xi) = \Gamma (\partial_x + X^{-1}S(y_0)) + ic(\xi; y_0),
\end{equation}
cf. \eqref{L-transform}, and its square $L(y_0,\xi)^2$.
We will work with the Sobolev scale $W^s(\R_+, H)$, defined in terms of
the interpolation scale $H^s\equiv H^s(S(y_0))$. As noted in Remark
\ref{different-sobolev-scales}, the interpolation scales $H^s(S(y_0))$ in general depend on 
the base point $y_0\in \R^b$. This does not play a role here, since in the
present section $y_0$ is fixed. 

\begin{proposition}\label{PropositionLSquarenullxi-2}
	Assuming the spectral Witt condition \eqref{Witt-condition}, 
	the mapping $$L(y,0)^2:W^{2,2}(\R_+, H) \to W^{0,0}(\R_+, H),$$ 
	is bijective with bounded inverse.
\end{proposition}

\begin{proof} 
	Consider the following commutative diagram:
	\begin{equation}
	\xymatrix@-1.75pc{W^{2,2} \ar@<1ex>[rrrr]^{L^2(y,0)} 
	\ar@<1ex>@{-->}[dddd]^{X^{-2}} & & & &
		W^{0,0}\ar@{=}[dddd] \ar@{-->}[llll]^{K(y,0)} \\
		& &  & & \\
		& & \circlearrowright & & \\
		& &  & & \\
		W^{2,0}\ar[rrrr]^{\tilde L^2} \ar[uuuu]^{X^2}& & & & W^{0,0},
		\ar@<1ex>@{-->}[llll]^{\tilde K }}
	\end{equation}
	where $\tilde{L}^2(y,0) = L^2(y,0)\circ X^2$ and the inverse maps $K$ and 
	$\tilde{K}$ are constructed as follows. 
	Let $\{\phi_{j}\}_{j\in\N}$ be an orthonormal base of $H$ 
	consisting of eigenvectors of $A^2(y)$ such that
	$A^2(y) \ \phi_{j} =  \nu_j^2 \ \phi_{j}$,
	where by convention we assume $\nu_j>0$. The geometric Witt condition 
	\eqref{Witt-condition} implies $\nu_j > \frac{3}{2}$ and by discreteness
	we conclude
	\begin{equation}\label{nu-j-delta}
	\exists \, \delta > 0 \ \forall \, j \in \N: \ \nu_j \geq \frac{3}{2}+ \delta.
	\end{equation}
	For any $j\in \N$ we define $E_j := \langle 
	\phi_{j}\rangle$. For any 
	$g \in L^2(\R_+)$ the equation 
	$L^2(y,0)  \ f \cdot \phi_{j} = g \cdot \phi_j \in L^2(\R_+, E_j)$ reduces to a 
	scalar equation	
	\begin{equation}\label{eq1}
	\Bl- \b_x^2 + \frac{1}{x^2} \Bl
	\nu_j^2-\frac{1}{4}\Br \Br  f  = g.	
	\end{equation}
	The fundamental solutions for that equation are
	\begin{equation}
	\psi^+_{\nu_j} (x) = x^{\nu_j + \frac{1}{2}}, \ \text{and} \ 
	\psi^-_{\nu_j} (x) = x^{-\nu_j +\frac{1}{2}}.
	\end{equation}
	In view of \eqref{nu-j-delta}, neither of them lies in $W^{2,2}(\R_+)$
	and hence $L^2(y,0)$ is injective on $W^{2,2}(\R_+,H)$. It remains to prove surjectivity.
	The fundamental solutions $\psi^\pm_{\mu_j}$ yield a 
	solution of the equation \Eqref{eq1} with
	\begin{equation}
	f(x) = \int_{0}^{\infty} k_j(x,y) g(y) dy  =: K_j g,
	\end{equation}
	where $K_j$ is an integral operator with the integral kernel 
	\begin{equation}\label{k-j-bessel-0}
	k_j(x,y) = \left\{\begin{array}{cc}
	\frac{1}{2\nu_j} \bl \frac{y}{x}\br^{\nu_j} (xy)^{\frac{1}{2}}, & y\leq x,\\
	\frac{1}{2\nu_j} \bl \frac{y}{x}\br^{-\nu_j} (xy)^{\frac{1}{2}}, & x\leq y.
	\end{array} \right.
	\end{equation}
	Accordingly, a solution of the scalar equation for any $\tilde g \in L^2(\R_+)$
	$$(L^2(y,0) \circ X^2)  \ \tilde f \cdot \phi_{j}= \tilde g \cdot \phi_{j} \in L^2(\R_+, E_j),$$ is given 
	in terms of $\tilde K_j = X^{-2} \circ K_j$ by $\tilde f = \tilde K_j \tilde g$. 
	
	The integral operators $\tilde K_j$ have been studied in Proposition 
	\ref{Schur-1-1}, which proves in view of \eqref{nu-j-delta}
	that for each $E_j$ the restriction $\tilde{L}(y,0)|_{E_j}$ 
	admits an inverse
	\begin{equation}
	\tilde K_j:W^{0,0}(\R_+,E_j) \to W^{2,0}(\R_+,E_j)
	\end{equation}
	with norm bounded uniformly in $j\in \N$. Equivalently, the restriction $L(y,0)|_{E_j}$ 
	admits an inverse
	\begin{equation}
	K_j : W^{0,0}(\R_+,E_j) \to W^{2,2}(\R_+,E_j)
	\end{equation}
	with norm bounded uniformly in $j\in \N$. By \eqref{K-0-2}, the operator norms of 
	$\nu_j \cdot (X\partial_x) \circ K_j$ and $\nu_j^2 \cdot K_j$ are bounded uniformly in $j$ as well.
	Hence there exists a bounded inverse
	\begin{equation}
	(L^2(y,0))^{-1} : W^{0,0}(\R_+,H) \to W^{2,2}(\R_+,H).
	\end{equation}
	This proves the statement.
\end{proof}

\begin{proposition}\label{PropositionLnotnullxi-2}
     Assume the spectral Witt condition \eqref{Witt-condition}. 
	Then for fixed parameters $(y,\xi) \in \R^b\times \R^b$, the operator 
	$L^2(y,\xi):W^{2,2}(\R_+, H)\to W^{0,0,-2}(\R_+,H)$ is injective with 
	a right-inverse $L^2(y,\xi)^{-1}: W^{0,0}(\R_+, H)\to W^{2,2}(\R_+,H)$, 
	bounded uniformly in the parameters $(y,\xi)$.
\end{proposition}

\begin{proof}
The case $\xi = 0$ has been established in Proposition \ref{PropositionLSquarenullxi-2}.
We proceed with the case $\xi \neq 0$. The commutator relations \Eqref{Eq1} imply that $A^2(y)$ 
and $c(\xi,y)^2$ may be simultaneously diagonalized and hence an orthonormal base 
$\{\phi_{j}\}_{j \in \N}$ of $H$ can be chosen such that
	\begin{equation}
		\begin{aligned}
			&A^2(y) \ \phi_{j} = \nu_j^2 \phi_{j},\; \text{where we 
			fix}\;\nu_j>0, \\
			&c(\hat{\xi}, y)^2 \phi_{j} =  
			\phi_{j},\;\text{where}\; \hat{\xi} = 
			\frac{\xi}{\|\xi\|}.
		\end{aligned}
	\end{equation}
	We write $E_j = \langle \phi_{j}\rangle$. Then, similar to 
	Proposition \ref{PropositionLSquarenullxi-2}, $L^2(y,\xi)$ 
	reduces over each $E_j$ to the scalar operator
	\begin{equation}
		L^2(y,\xi)|_{E_j} = -\b_x^2 + X^{-2} \Bl\nu_j^2 - \frac{1}{4}\Br + 
		\|\xi\|^2.
	\end{equation}
	The solutions to $L^2(y,\xi)|_{E_j} \phi = 0$ are given by linear 
	combination of modified Bessel-functions $\sqrt{x} I_{\nu_j}(\|\xi\| x)$
	and $\sqrt{x} K_{\nu_j}(\|\xi\| x)$, which are {\bf not} elements of $W^{2,2}(\R_+)$
	for any $j\in\N$ and any $\xi \not=0$. This proves injectivity of $L^2(y,\xi)$ on
	$W^{2,2}(\R_+,H)$.
	
	For the right-inverse we note the following commutative diagram
	\begin{equation}
	\xymatrix@-1.75pc{W^{2,2} \ar@<1ex>[rrrr]^{L^2(y,\xi)} 
		\ar@<1ex>@{-->}[dddd]^{X^{-2}} & & & &
		W^{0,0,-2}\ar@{=}[dddd] \ar@{-->}[llll]^{K_1(y,\xi)}\\
		& &  & &\\
		& & \circlearrowright & &\\
		& &  & &\\
		W^{2,0}\ar[rrrr]^{\tilde L^2} \ar[uuuu]^{X^2}& & & & W^{0,0,-2} 
		\ar@<1ex>@{-->}[llll]^{\tilde K_1 }}.
	\end{equation}
	The equation $[ L^2(y,\xi)\circ X^2|_{E_j}]f \cdot \phi_j= g \cdot \phi_j\in L^2(\R_+,E_j)$ 
	admits a solution
	\begin{equation}\label{L-inverse-kernel}
		\begin{aligned}
			X^{-2}\circ K_{j}(y,\xi) g &:= \int_{\R_+} 
			x^{-2}\  k_{j}(x,\tilde{x})\ g(\tilde{x}) 
			d\tilde{x},
		\end{aligned}
	\end{equation} 
	 where the kernel $k_{j}(x,\tilde{x})$ is 
	 \begin{equation}\label{k-j-bessel}
	 k_{j}(x,\tilde{x}) = \left\{ \begin{array}{cc}
	 (x\tilde{x})^{\frac{1}{2}} I_{\nu_j} (\|\xi\| \ \tilde{x}) 
	 K_{\nu_j}(\|\xi\| \ 
	 x), 
	 & \tilde{x}\leq x,\\
	 (x\tilde{x})^{\frac{1}{2}} I_{\nu_j} (\|\xi\| \ x) K_{\nu_j}(\|\xi\| \ 
	 \tilde{x}), 
	 & x\leq \tilde{x}.
	 \end{array}
	 \right.
	 \end{equation}
	 Therefore, by Proposition \ref{Schur-2-1}, $X^{-2}\circ 
	 K_{j}$ is bounded, uniformly in $j\in \N$ and $\xi > 0$. Then, 
	 in view of the uniform bounds \eqref{K2-statement}, 
	 $L^2(y,\xi) \circ X^2$ admits a right-inverse $X^{-2} \circ K(y,\xi):  W^{0,0}(\R_+,H) \to W^{2,0} 
	 (\R_+,H)$. Equivalently, $L^2(y,\xi)$ admits a right-inverse $K(y,\xi):  W^{0,0}(\R_+,H) \to W^{2,2} 
	 (\R_+,H)$, which proves the statement in view of continuity at $\xi = 0$.
\end{proof}

\begin{corollary}\label{PropositionLnotnullxi}
Assume the spectral Witt condition \eqref{Witt-condition}.
     Then $$L(y,\xi):W^{1,1}(\R_+,H)\to W^{0,0,-1}(\R_+,H),$$ is 
	injective with right-inverse $L(y,\xi)^{-1}: W^{0,0}(\R_+, H)\to W^{1,1}(\R_+,H)$, 
	bounded uniformly in $(y,\xi)$.
\end{corollary}

\begin{proof}
      The commutator relations \Eqref{Eq1} imply that $S,\Gamma$ and $ic(\xi)$ 
	may be simultaneously diagonalized and hence an orthonormal base 
	$\{\phi_{j,\pm}\}$ of $H$ can be chosen such that
		\begin{equation}
	\begin{aligned}
	S\phi_{j,\pm} &= \pm \mu_j \phi_{j,\pm},\; \text{where we fix}\;\mu_j>0, \\
	ic(\hat{\xi}, y) \phi_{j,\pm} &= \pm \phi_{j,\pm},\;\text{where}\; \hat{\xi} = 
	\frac{\xi}{\|\xi\|}\\
	\Gamma\phi_{j,\pm} &= \pm\phi_{j,\mp}.
	\end{aligned}
	\end{equation}
	We define $E_j = \langle \phi_{j,+};\phi_{j,-}\rangle$. Then $L(y,\xi)$ reduces over each $E_j$ to 
	\begin{equation*}
	L(y,\xi)|_{E_j} = \left(
	\begin{array}{cc}
	0 & -I\\
	I & 0
	\end{array}\right)\left[
	\left(\begin{array}{cc}
	\partial_x & 0 \\
	0 & \partial_x
	\end{array}\right)
	+ X^{-1} \left(\begin{array}{cc}
	\mu_j & 0\\
	0 & -\mu_j
	\end{array}\right)
	\right] + \left(
	\begin{array}{cc}
	\|\xi\| & 0 \\
	0 & -\|\xi\|
	\end{array}
	\right).
	\end{equation*}
	Like in \cite[Lemma 3.10]{Albin-Jesse}, solutions to $L(y,\xi)|_{E_j} \phi = 0$ 
	are given by linear combination of modified Bessel-functions,
	which are {\bf not} elements of $W^{1,1}$ for any $j\in\N$ and any $\xi 
	\not=0$. Same can be checked explicitly for $\xi = 0$. This proves 
	injectivity of $L(y,\xi)$.
The right-inverse is obtained by
\begin{equation}
L(y,\xi)^{-1} := L(y,\xi) \circ \left(L^2(y,\xi)\right)^{-1}: L^2(\R_+,H) \to 
W^{1,1} (\R_+,H),
\end{equation}
where the composition is well-defined for $\xi = 0$ by Proposition 
\ref{PropositionLSquarenullxi-2}, and for $\xi \neq 0$ by the fact that 
$\left(L^2(y,\xi)\right)^{-1}$ maps $L^2(\R_+, H)$ to $W^{2,2} \cap L^2(\R_+,H)$, since
$$
L(y,0) \circ \left(L^2(y,\xi)\right)^{-1} = \textup{Id} - \| \xi \|^2 \cdot \left(L^2(y,\xi)\right)^{-1}.
$$
\end{proof}

In the Corollary \ref{PropositionLnotnullxi} there is a certain overlap with the work of Albin and 
Gell-Redman \cite{Albin-Jesse}, where in \cite[Lemma 3.10]{Albin-Jesse}
they assert invertibility of $L(y,\xi)$ for $\xi\neq 0$, 
and do not prove uniform bounds for the inverse. Here,
we invert $L(y,\xi)$ for all $\xi \in \R^b$ and establish 
uniform bounds for the inverse.

\section{Parametrices for generalized Dirac and Laplace operators}

We define subspaces of functions with compact support in $[0,1]$
 \begin{equation*}
 \begin{split}
 &W^\bullet_{\rm comp}(\R_+, H):= \{\phi u \mid  
 u \in W^\bullet (\R_+, H), \phi \in C^\infty_0[0,\infty), \supp \phi \subset [0,1]\}, \\
 &W^\bullet_{\rm comp}(\R_+\times \R^b, H):= \{\phi u \mid   
 u \in W^\bullet (\R_+\times \R^b, H), \phi \in C^\infty_0([0,\infty) \times \R^b), \\
 &   \qquad  \qquad \qquad \qquad \qquad \supp \phi \subset [0,1] \times \R^b\}.
 \end{split}
 \end{equation*}
Subspaces of weighted Sobolev scales, consisting of functions with compact support 
in $[0,1]$ and $[0,1] \times \R^b$ as above, are denoted analogously. The Sobolev 
scales are defined in terms of the interpolation scales $H^s\equiv H^s(S(y_0))$, which a
priori depend on the base point $y_0 \in \R^b$. This does not play a role here, since in the
present section $y_0$ is fixed. 

\begin{proposition}\label{comp-asymptotics-infinity-prop} Assume the spectral Witt condition 
	\eqref{Witt-condition}. Then there exists $\delta>0$ such that 
	for any $u \in W^0_{\rm comp}(\R_+, H)$ and any $\xi \in \R^b$
	\begin{equation*}
	\| L(y_0,\xi)^{-1} u(x) \|_H = O(x^{-1-\delta}), \quad
	\| L^2(y_0,\xi)^{-1} u(x) \|_H = O(x^{-1-\delta}), \quad x\to \infty. 
	\end{equation*}
	In particular, $L(y_0,\xi)^{-1} u$ and $L^2(y_0,\xi)^{-1} u$ are both in $L^2(\R_+,H)$.
	Here, $\| \cdot \|_H$ denotes the norm of the Hilbert space $H$.
\end{proposition}

\begin{proof} Consider $u \in W^0_{\rm comp}(\R_+, H)$.
Note first that $L(y_0,\xi)^{-1} u \in W^{1,1} (\R_+, H)$ and
$L^2(y_0,\xi)^{-1} u \in W^{2,2} (\R_+, H)$ by Proposition \ref{PropositionLnotnullxi-2} and
Corollary \ref{PropositionLnotnullxi}. By the characterization \eqref{Sobolev-space-definition}
of Sobolev scales and the Sobolev embedding $H^1_e(\R_+) \subset C(0,\infty)$ into continuous
functions, $L(y_0,\xi)^{-1} u$ and $L^2(y_0,\xi)^{-1} u$ are
continuous on $(0,\infty)$ with values in $H$, and in that sense their 
pointwise evaluations are well-defined. Recall
\begin{equation*}
L^2(y_0,\xi) = -\b_x^2 + X^{-2}\left(A^2(y_0) - \frac{1}{4}\right) + c(\xi,y_0)^2.
\end{equation*}
By the spectral Witt condition, $\textup{Spec} \, A(y_0) \cap [0,\frac{3}{2}] = \varnothing$
and by discreteness of the spectrum there exists $\delta>0$ such that
\begin{equation}\label{spec-delta}
\textup{Spec} \, A(y_0) \cap [0,\frac{3}{2}+\delta) = \varnothing.
\end{equation}
The integral kernel of $L^2(y_0,\xi)^{-1}$ is given 
in terms of \eqref{k-j-bessel} for $\xi \neq 0$ and \eqref{k-j-bessel-0} for $\xi = 0$. 
In view of \eqref{spec-delta}, in both cases, the asymptotics $\| L^2(y_0,\xi)^{-1} u(x) \|_H= O(x^{-1-\delta})$ as 
$x\to \infty$ follows from Corollary \ref{comp-asymptotics-infinity}. The asymptotics
of $\| L(y_0,\xi)^{-1} u(x)\|_H $ now follows also by Corollary \ref{comp-asymptotics-infinity},
once we observe that $L(y_0,\xi)^{-1} u = L(y_0,\xi) (L^2(y_0,\xi)^{-1} u) \in W^{1,1} (\R_+, H)$.
\end{proof}

\begin{theorem}\label{TheoremParametrix}
	Consider $u\in C_0^\infty(\R_+\times \R^b,H^\infty)$ and denote its Fourier 
	transform on $\R^b$ by $\hat{u} (\xi)$. Fix $y_0 \in \R^b$ and 
	consider a generalized Dirac operator $D_{y_0}$ satisfying the spectral Witt condition 
	\eqref{Witt-condition}. We define 
	\begin{equation}
	Qu(y):= \int_{\R^b} e^{i \langle y, \xi\rangle} L(y_0,\xi)^{-1} \hat{u} (\xi) 
	\dj \xi, \quad \dj \xi := \frac{d \xi}{(2\pi)^b}.
	\end{equation}
	Then $Q$ is a right-inverse to $D_{y_0}$ and defines a bounded operator
	\begin{equation}
	Q:W^{0}_{\rm comp}(\R_+\times \R^b,H) \subset W^{0} \to X\cdot W^{1}(\R_+\times\R^b,H) = W^{1,1}.
	\end{equation}
	\end{theorem}

\begin{proof}
By the Plancherel theorem we find for any $u\in C_0^\infty(\R_+\times \R^b,H^\infty)$
\begin{align*}
\|X^{-1} Q u \|^2_{L^2(\R_+\times \R^b_y, H)}
&= \| X^{-1} L(y_0,\cdot )^{-1} \hat{u} \|^2_{L^2(\R_+\times \R^b_\xi, H)}
\\ &= \int_{\R^b} \|X^{-1} L(y_0,\xi )^{-1} \hat{u}(\xi) \|^2_{L^2(\R_+, H)} \dj \xi.
\end{align*}
By Corollary \ref{PropositionLnotnullxi}, the operator $X^{-1} L(y_0,\xi )^{-1}$
defines a bounded map from $L^2(\R_+, H)$ to itself, with the operator norm bounded
uniformly in $\xi \in \R^b$. Denote its uniform bound by $C>0$ and compute
again by Plancherel theorem
\begin{align*}
\|X^{-1} Q u \|^2_{L^2(\R_+\times \R^b_y, H)}
&= \int_{\R^b} \|X^{-1} L(y_0,\xi )^{-1} \hat{u}(\xi) \|^2_{L^2(\R_+, H)} \dj \xi
\\ & \leq C \int_{\R^b} \|\hat{u}(\xi) \|^2_{L^2(\R_+, H)} \dj \xi = 
C \| u \|^2_{L^2(\R_+\times \R^b_y, H)}.
\end{align*}
Consequently, $Q: L^2(\R_+\times \R^b,H) \to X\cdot L^2(\R_+\times \R^b,H) = W^{0,1}$
is bounded. 

Furthermore, by Corollary \ref{PropositionLnotnullxi} 
the operators $(X\partial_x) \circ X^{-1} L(y_0,\xi )^{-1}$ and  $S \circ X^{-1} L(y_0,\xi )^{-1}$
are bounded on $L^2(\R_+, H)$ and by the same argument as before $(X\partial_x) \circ Q$ and 
$S \circ Q$ define bounded operators from $L^2$ to $W^{0,1}$. In order to prove the 
statement, it remains to establish boundedness of $(X \partial_y) \circ Q: L^2 \to W^{0,1}$.

For $u \in L^2_{\rm comp}(\R_+\times \R^b,H)$ with compact support in $[0,1] \times \R^b$, its Fourier 
transform $\hat{u} (\xi)$ in the $\R^b$ component, is still an element of $L^2_{\rm comp}(\R_+,H)$ 
with compact support in $[0,1]$. By Corollary \ref{PropositionLnotnullxi} there exists a preimage 
$v = L(y_0,\xi)^{-1} \hat{u} (\xi)$,
and by Proposition \ref{comp-asymptotics-infinity-prop} its norm in $H$ is $O(x^{-1-\delta})$
as $x\to \infty$ for some $\delta>0$. In particular, $v \in L^2(\R_+,H)$.
We compute using commutator relations \Eqref{Eq1},
	\begin{equation}
	\begin{aligned}
	\langle \ L(y_0,\xi) v, &L(y_0,\xi) v \ \rangle_{L^2} =\langle L(y_0,\xi)^2 v , 
	v \rangle_{L^2}\\
	&=\langle (-\partial_x^2 + X^{-2}(S(y_0)^2+S(y_0))) v , v \rangle_{L^2} + \|\xi 
	\|^2\cdot \| v \|^2_{L^2}\\
	&=\|(\partial_x+X^{-1}S(y_0)) v \|^2_{L^2} + \|\xi \|^2\cdot \| v \|^2_{L^2}\\
	&\geq \|\xi \|^2\cdot \| v \|^2_{L^2},
	\end{aligned}
	\end{equation}
	where boundary terms at $x=0$ do not arise due to the weight
	$x$ in $W^{1,1} = X W^{1,0}$. Boundary terms at $x=\infty$ 
	do not arise since $\| v(x) \|_H =O(x^{-1-\delta})$ as $x\to \infty$ for some $\delta>0$. 
	We arrive at the following estimate
	\begin{equation}
	\frac{\|L(y_0,\xi)^{-1} \hat{u} (\xi)\|_{L^2}}{\| \hat{u} (\xi) \|_{L^2}} = 
	\frac{\|L(y_0,\xi)^{-1}L(y_0,\xi) v \|_{L^2}}{\| L(y_0,\xi) v \|_{L^2}} = 
	\frac{\| v \|_{L^2}}{\|L(y_0,\xi) v \|_{L^2}} \leq \|\xi\|^{-1}.
	\end{equation}
	By continuity at $\xi = 0$ we conclude for some constant $C>0$
	\begin{equation}
	\|L(y_0,\xi)^{-1} \hat{u} (\xi)\|_{L^2} \leq C \cdot (1+\|\xi\|) ^{-1} \| \hat{u} (\xi)\|_{L^2}.
	\end{equation}
	We may now estimate for any $u \in W^{0}_{\rm comp}(\R_+\times \R^b,H)$
\begin{align*}
\|X^{-1} (X \partial_{y_i}) Q u \|^2_{L^2(\R_+\times \R^b_y, H)}
&= \int_{\R^b} \| \xi_i \cdot L(y_0,\xi )^{-1} \hat{u}(\xi) \|^2_{L^2(\R_+, H)} \dj \xi
\\ &\leq C \int_{\R^b} \frac{\| \xi \|^2}{\| 1+\xi \|^2} \| \hat{u}(\xi) \|^2_{L^2(\R_+, H)} \dj \xi
\\ &\leq C \| u \|^2_{L^2(\R_+\times \R^b_y, H)}.
\end{align*}
This finishes the proof.
\end{proof}

We point out that it is precisely the fact that we have established invertibility of $L(y_0,\xi)$
for any $\xi \in \R^b$ instead of $\xi \neq 0$, which allows us to write down the parametrix
$Q$ explicitly using Fourier transform and establish its mapping properties as a simple 
consequence of the Plancherel theorem. In case of $L(y_0,\xi)$ being invertible only for 
$\xi \neq 0$ the parametrix construction needs to take care of a singularity at $\xi=0$ via
cutoff functions, in which case one cannot deduce its mapping properties by a simple 
application of the Plancherel theorem and is forced to employ an operator valued version
of the theorem by Calderon and Vaillancourt \cite{CaVa}. \medskip

We conclude with construction of a parametrix for $D_{y_0}^2$,
cf. Theorem \ref{TheoremParametrix}.

\begin{theorem}\label{TheoremParametrix-2}
      Assume the spectral Witt condition \eqref{Witt-condition}.
	Consider $u\in C_0^\infty(\R_+\times \R^b,H^\infty)$ and denote its Fourier 
	transform on $\R^b$ by $\hat{u} (\xi)$. Fix $y_0 \in \R^b$ and 
	consider the square $D^2_{y_0}$ of a generalized Dirac operator. We define 
	\begin{equation}
	Q^2u(y):= \int_{\R^b} e^{i \langle y, \xi\rangle} (L^2(y_0,\xi))^{-1} \hat{u} (\xi) 
	\dj \xi, \quad \dj \xi := \frac{d \xi}{(2\pi)^b}.
	\end{equation}
	Then $Q^2$ is a right-inverse to $D^2_{y_0}$ and defines a bounded operator
	\begin{equation}
	Q^2:W^{0}_{\rm comp}(\R_+\times \R^b,H) \subset W^{0} \to X^2\cdot 
	W^{2}(\R_+\times\R^b,H) = W^{2,2}.
	\end{equation}
	\end{theorem}

\begin{proof}
By the Plancherel theorem we find for any $u\in C_0^\infty(\R_+\times \R^b,H^\infty)$
\begin{align*}
\|X^{-2} \circ Q^2 u \|^2_{L^2(\R_+\times \R^b_y, H)}
&= \| X^{-2} \circ (L^2(y_0,\cdot ))^{-1} \hat{u} \|^2_{L^2(\R_+\times \R^b_\xi, H)}
\\ &= \int_{\R^b} \|X^{-2} \circ (L^2(y_0,\xi ))^{-1} \hat{u}(\xi) \|^2_{L^2(\R_+, H)} \dj \xi.
\end{align*}
By Proposition \ref{PropositionLnotnullxi-2}, the operator $X^{-2} (L^2(y_0,\xi ))^{-1}$
defines a bounded map from $L^2(\R_+, H)$ to itself, with the operator norm bounded
uniformly in $\xi \in \R^b$. Denote its uniform bound by $C>0$ and compute
again by Plancherel theorem
\begin{align*}
\|X^{-2} \circ Q^2 u \|^2_{L^2(\R_+\times \R^b_y, H)}
&= \int_{\R^b} \|X^{-2} \circ (L^2(y_0,\xi ))^{-1} \hat{u}(\xi) \|^2_{L^2(\R_+, H)} \dj \xi
\\ & \leq C \int_{\R^b} \|\hat{u}(\xi) \|^2_{L^2(\R_+, H)} \dj \xi = 
C \| u \|^2_{L^2(\R_+\times \R^b_y, H)}.
\end{align*}
Consequently, $Q^2: L^2(\R_+\times \R^b,H) \to X^2\cdot L^2(\R_+\times \R^b,H) = W^{0,2}$
is bounded. Furthermore, by Proposition \ref{PropositionLnotnullxi-2} we find for any
$$ V_1, V_2 \in \mathcal{K} := \{(X\partial_x), S\},$$
that the operators $V_1 \circ X^{-2} \circ (L^2(y_0,\xi ))^{-1}$ and 
$V_1 \circ V_2 \circ X^{-2} \circ (L^2(y_0,\xi ))^{-1}$
are bounded on $L^2(\R_+, H)$. By the same argument as before, 
$V_1 \circ Q^2$ and $V_2 \circ V_2 \circ Q^2$
define bounded operators from $L^2(\R_+\times \R^b,H)$ to 
$W^{0,2}(\R_+\times \R^b,H)$. In order to prove the 
statement, it remains to establish boundedness of $(X \partial_y) \circ V_1 \circ Q^2$ 
and $(X \partial_y)^2 \circ Q^2$
as maps from $L^2(\R_+\times \R^b,H)$ to $W^{0,2}(\R_+\times \R^b,H)$.
\medskip

For $u \in W^{0}_{\rm comp}(\R_+\times \R^b,H)$ with compact support in $[0,1] \times \R^b$, its Fourier 
transform $\hat{u} (\xi)$ in the $\R^b$ component, is still an element of $W^{0}_{\rm comp}(\R_+,H)$
with compact support in $[0,1]$.
By Proposition \ref{PropositionLnotnullxi-2}, $v = L^2(y_0,\xi)^{-1} \hat{u} (\xi) \in 
W^{2,2}(\R_+,H)$. By Proposition \ref{comp-asymptotics-infinity-prop}, $\| v(x) \|_H = O(x^{-1-\delta})$
as $x\to \infty$ for some $\delta>0$. In particular, $v \in L^2(\R_+,H)$.
We compute using commutator relations \Eqref{Eq1},
	\begin{equation}
	\begin{aligned}
	\langle \ L(y_0,\xi) v, \, &L(y_0,\xi) v \ \rangle_{L^2} =\langle L(y_0,\xi)^2 v , 
	v \rangle_{L^2}\\
	&=\langle (-\partial_x^2 + X^{-2}(S(y_0)^2+S(y_0))) v , v \rangle_{L^2} + \|\xi 
	\|^2\cdot \| v \|^2_{L^2}\\
	&=\|(\partial_x+X^{-1}S(y_0)) v \|^2_{L^2} + \|\xi \|^2\cdot \| v \|^2_{L^2}\\
	&\geq \|\xi \|^2\cdot \| v \|^2_{L^2},
	\end{aligned}
	\end{equation}
	where there are no boundary terms after integration by parts.
	More precisely, boundary terms at $x=0$ do not arise due to the weight
	$x^2$ in $W^{2,2} = X^2 W^{2,0}$. Boundary terms at $x=\infty$ 
	do not arise since $\| v(x) \|_H = O(x^{-1-\delta})$ as $x\to \infty$.
	
	By Proposition \ref{comp-asymptotics-infinity-prop}, $L(y_0,\xi) v \in W^{1,1}(\R_+,H)$
	with the asymptotic expansion $\| L(y_0,\xi) v(x) \|_H = O(x^{-1-\delta})$ as $x\to \infty$
	as well. Hence, in the estimates above, we can replace $v$ with $w = L(y_0,\xi) v$
	and still conclude
	\begin{equation}
	\langle \ L(y_0,\xi) w, L(y_0,\xi) w \ \rangle_{L^2} \geq \|\xi \|^2\cdot \| w \|^2_{L^2}.
	\end{equation}
	We arrive at the following estimate
	\begin{equation*}
	\frac{\|L^2(y_0,\xi)^{-1} \hat{u} (\xi)\|_{L^2}}{\| \hat{u} (\xi) \|_{L^2}} = 
	\frac{\| v \|_{L^2}}{\|L^2(y_0,\xi) v \|_{L^2}} \leq 
	\|\xi\|^{-1} \frac{\| v \|_{L^2}}{\|L(y_0,\xi) v \|_{L^2}}
	\leq \|\xi\|^{-2}.
	\end{equation*}
	By continuity at $\xi = 0$ we conclude for some constant $C>0$
	\begin{equation}
	\|L^2(y_0,\xi)^{-1} \hat{u} (\xi)\|_{L^2} \leq C \cdot (1+\|\xi\|) ^{-2} \| \hat{u} (\xi)\|_{L^2}.
	\end{equation}
	We may now estimate for any $u \in W^{0}_{\rm comp}(\R_+\times \R^b,H)$
\begin{align*}
\|X^{-2} (X \partial_{y_i}) (X \partial_{y_j}) Q^2 u \|^2_{L^2(\R_+\times \R^b_y, H)}
&= \int_{\R^b} \| \xi_i \xi_j \cdot L^2(y_0,\xi )^{-1} \hat{u}(\xi) \|^2_{L^2(\R_+, H)} \dj \xi
\\ &\leq C \int_{\R^b} \frac{\| \xi \|^2}{(1+\| \xi \|)^2} \| \hat{u}(\xi) \|^2_{L^2(\R_+, H)} \dj \xi
\\ &\leq C \| u \|^2_{L^2(\R_+\times \R^b_y, H)}.
\end{align*}
Similar estimate holds for $(X \partial_y) V \, Q^2 u$ with $V \in \mathcal{K}$.
This finishes the proof.
\end{proof}

\section{Minimal domain of a Dirac Operator on an abstract edge}\label{min-domain-Dirac-section}

We now employ the previous parametrix construction in order
to deduce statements on the minimal and maximal domains of $D_{y_0}$ 
and consequently for $D$. Recall $H=H(S(y_0))$ and the basic definitions of minimal and maximal domains.
As noted in Remark \ref{different-sobolev-scales},
the interpolation scales $H^s(S(y))$ and $H^s(S(y_0))$ coincide for $0\leq s \leq 1$. 

\begin{definition}\label{minmax} 
The maximal and minimal domain of $D$ are defined as follows:
	\begin{equation*}
	\begin{aligned}
	\sD(D_{\max})&:=\{u\in L^2(\R_+\times \R^b,H)\ | \ 
	Du\in L^2(\R_+\times\R^b,H) \}\\
	\sD(D_{\min})&:=\{u\in\sD(D_{\max}) \ | \ \exists (u_n)\subset 
	C_0^\infty(\R_+\times \R^b,H^\infty) \\ & \qquad \qquad \qquad \qquad \text{with}\; u_n \stackrel{L^2}{\to} u,\; 
	Du_n \stackrel{L^2}{\to} Du    \}.
	\end{aligned}
	\end{equation*}
\end{definition}

Using smooth cutoff functions we define localized versions of domains:
\begin{equation}\label{localized-domains}
\begin{aligned}
\sD_{\rm comp}(D_{\max})&:=\{\varphi u\ | u\in \sD(D_{\max}), 
\varphi \in C^\infty_0([0,\infty) \times \R^b) \},\\
\sD_{\rm comp}(D_{\min})&:=\{\varphi u\ | u\in \sD(D_{\min}), 
\varphi \in C^\infty_0([0,\infty) \times \R^b)\},\\
W^{s,\delta}_{\rm \comp}(\R_+\times \R^b,H)&:=\{\varphi u\ | u\in X^\delta W^s, 
\varphi \in C^\infty_0([0,\infty) \times \R^b)\},
\end{aligned}
\end{equation}
where in each case we additionally require\footnote{Restriction of the support to be in 
$[0,1] \times \R^b$ is necessary to achieve uniformity of the estimates in Corollary 
\ref{comp-asymptotics-infinity} and for the consequence in Proposition \ref{comp-asymptotics-infinity-prop}
to hold.} $\supp \varphi \subset [0,1] \times \R^b$.
One checks directly from the definitions
\begin{equation}
\sD_{\rm comp}(D_{\max / \min}) \subseteq \sD(D_{\max/\min}).
\end{equation}
The maximal and minimal domains $\sD(D_{y_0, \, \max}), \sD(D_{y_0, \, \min})$
and their respective localized versions $\sD_{\rm comp}(D_{y_0, \, \max}), \sD_{\rm comp}(D_{y_0, \, \min})$
are defined analogously.
\begin{lemma}\label{LemmaDcompmaxmin}
	$\sD_{\rm comp}(D_{y_0, \, \max / \min}) \subseteq W^{1,1}_{\rm comp}(\R_+\times 
	\R^b,H)$.
\end{lemma}

\begin{proof}
Since $\sD_{\rm comp}(D_{y_0, \, \min}) \subseteq
\sD_{\rm comp}(D_{y_0, \, \max})$, it suffices
to show 
\begin{equation*}
\sD_{\rm comp}(D_{y_0, \, \max}) \subseteq W^{1,1}_{\rm comp}(\R_+\times \R^b,H).
\end{equation*}
Note that the differential expression $D_{y_0}$ induces two mappings
\begin{align*}
&D_{y_0} : \sD (D_{y_0, \, \max}) \to L^2(\R_+\times \R^b,H), \\
&D_{y_0} : W^{1,1}(\R_+\times \R^b,H) \to L^2(\R_+\times \R^b,H),
\end{align*}
where the former is an unbounded self-adjoint operator in the Hilbert
space $L^2(\R_+\times \R^b,H)$, and the latter is a bounded operator 
between Sobolev spaces\footnote{Note that 
$W^{1,1}(\R_+\times \R^b,H) \nsubseteq L^2(\R_+\times \R^b,H)$.}.
Theorem \ref{TheoremParametrix} provides the right inverse 
$Q:L^2_{\rm comp}(\R_+\times \R^b,H)\to 
W^{1,1}(\R_+\times \R^b,H)$ to the latter mapping, but not to the former. 
More precisely, we only have
\begin{equation*}
\forall \, u \in L^2_{\rm comp}(\R_+\times \R^b,H): \quad D_{y_0} (Q u) = u.
\end{equation*}
The same holds for the formal adjoints $D^t_{y_0}$ and $Q^t$ in $L^2(\R_+\times \R^b,H)$
\begin{align*}
&D^t_{y_0} : W^{1,1} \to L^2, \quad Q^t:L^2_{\rm comp}\to W^{1,1}\\
&\forall \, u \in L^2_{\rm comp}(\R_+\times \R^b,H): \quad D^t_{y_0} (Q^t u) = u.
\end{align*}
Consider $u \in \sD_{\rm comp}(D_{y_0, \, \max})$ and a test function 
$\phi \in C_0^\infty(\R_+\times \R^b,H^\infty)$. We fix a smooth cutoff function 
$\psi \in C_0^\infty(\R_+\times \R^b,H^\infty)$, such that $\psi \equiv 1$ on 
$\supp u \cup \supp \phi$. We compute with $L^2=L^2(\R_+\times \R^b,H)$
\begin{align*}
\langle u, \phi \rangle_{L^2} &= \langle u, \psi  D^t_{y_0} (Q^t \phi ) \rangle_{L^2}
\\ &= \langle u, [\psi, D^t_{y_0}] (Q^t \phi ) \rangle_{L^2} + \langle u, D^t_{y_0} \psi (Q^t \phi ) \rangle_{L^2}
\end{align*}
Note that $\supp \, [\psi, D^t_{y_0}]$ is by construction disjoint from $\supp u$ and 
consequently the first summand above is zero. Using $u \in \sD_{\rm comp}(D_{y_0, \, \max})$
we can integrate by parts and conclude
\begin{align*}
\langle u, \phi \rangle_{L^2} &= \langle u, D^t_{y_0} \psi (Q^t \phi ) \rangle_{L^2}
\\ &= \langle D_{y_0} u, \psi (Q^t \phi ) \rangle_{L^2} = \langle D_{y_0} u, Q^t \phi \rangle_{L^2}
\\ &= \langle Q D_{y_0} u, \phi \rangle_{L^2}.
\end{align*}
We conclude that $u = Q(D_{y_0}u)$ as distributions. 
By Theorem \ref{TheoremParametrix} 	
\begin{equation}\label{QD}
	u = Q(D_{y_0}u)\in W^{1,1}_{\rm comp}(\R_+\times \R^b,H).
	\end{equation}
\end{proof}

\begin{corollary}\label{Corollaryequalitymaxmin}
	$\sD_{\rm comp}(D_{y_0, \, \min}) = \sD_{\rm comp}(D_{y_0, \, \max}) = W^{1,1}_{\rm 
	comp}(\R_+\times \R^b,H)$ .
\end{corollary}

\begin{proof}
By Lemma \ref{LemmaDcompmaxmin} it suffices to show that 
$W^{1,1}_{\rm comp}(\R_+\times \R^b,H)$ is included in $\sD_{\rm comp}(D_{y_0, \, \min})$.
Note that $D_{y_0}: W^{1,1}_{\rm comp}(\R_+\times \R^b,H) 
\to L^2(\R_+\times \R^b,H)$ is continuous, and $C_0^\infty(\R_+\times \R^b,H^\infty)
\subset W^{1,1}_{\rm comp}(\R_+\times \R^b,H)$ is dense.
Consider $u\in W^{1,1}_{\rm comp}(\R_+\times \R^b,H)$ and some $(u_n)\subset 
C^\infty_0(\R_+\times \R^b,H^\infty)$ such that $u_n \stackrel{W^{1,1}}{\rightarrow} u$. 
By continuity, $D_{y_0}u_n \to D_{y_0} u$ in $L^2$. Hence by definition, 
$u\in \sD_{\rm comp}(D_{y_0, \, \min})$.
\end{proof}

Now we want to extend this statement to a perturbation of $D_{y_0}$
\begin{equation}
P = \Gamma(\partial_x + X^{-1}S(y)) + T + V 
:=D_{y_0} + D_{1,y_0}
\end{equation}
where $V:W^{1,1}(\R_+ \times \R^b, H) \to 
W^{0,1}(\R_+ \times \R^b, H)$ is a bounded linear operator, preserving
compact supports and usually referred to as a higher order term.

\begin{theorem}\label{TheoremDcompPmimPmax}
Assume in addition to the spectral Witt condition \eqref{Witt-condition} that 
\begin{equation}\label{smooth-family}
\partial_y S(y)(|S(y_0)|+1)^{-1}
\end{equation}
are bounded operators on $H$ for any $y, y_0\in \R^b$.
Then 
\begin{equation}
\sD_{\rm comp}(P_{\min}) = \sD_{\rm comp}(P_{\max}) = 
W^{1,1}_{\rm comp}(\R_+ \times \R^b, H).
\end{equation}
\end{theorem}

\begin{proof} 
\textbf{Step 1:} Consider $u\in C_0^\infty((0,\infty)\times \R^b, H^\infty)$ and 
smooth cutoff functions $\phi, \psi \in C_0^\infty([0,\infty)\times \R^b)$ 
taking values in $[0,1]$, such that $\supp\phi \subset [0,\epsilon) 
\times B_\epsilon(y_0)$ and $\psi \restriction \supp u \equiv 1$.
We compute using \eqref{QD}
	\begin{equation}
	\begin{aligned}
		\|\phi D_{1,y_0} u\|_{L^2} &= \|\phi D_{1,y_0} Q D_{y_0} u\|_{L^2}
		= \|\phi D_{1,y_0} Q \psi D_{y_0} u\|_{L^2}\\
		&\leq \|\phi D_{1,y_0} Q \psi\|_{L^2\to L^2} \cdot \|D_{y_0} 
		u\|_{L^2}.
	\end{aligned}
	\end{equation}
	In order to estimate the norm of $\phi D_{1,y_0} Q \psi$, note that
	\begin{equation}
	\begin{aligned}
	D_{1,y_0} Q  & = \Gamma \, X^{-1} \left(S(y)-S(y_0)\right) Q + \left(T-T_{y_0}\right) Q +  V Q
	\\ & = \Gamma \, X^{-1} (y-y_0) \int_0^1 \frac{\partial S}{\partial 
	t}\left(y_0 + t(y-y_0)\right) dt\ Q
	\\ & + (y-y_0) \int_0^1 \frac{\partial T}{\partial t}\left(y_0 + 
	t(y-y_0)\right)dt \ Q + VQ.
	\end{aligned}
	\end{equation}
	In view of the assumption \eqref{smooth-family} and boundedness of 
	the higher order term $V:W^{1,1} \to W^{0,1}$ we conclude from 
	Theorem \ref{TheoremParametrix} that
	\begin{equation*}
	X^{-1} \frac{\partial S}{\partial t}\left(y_0 + t(y-y_0)\right) Q \psi, 
	\quad \frac{\partial T}{\partial t}\left(y_0 + t(y-y_0)\right) \circ Q \psi, 
	\quad X^{-1} V \circ Q \psi
	\end{equation*}
	are bounded operators on $L^2(\R_+ \times \R^b, H)$ with bound uniform in $t\in [0,1]$
	and $\psi$. Hence we conclude for some unform constant $C>0$
	\begin{equation}
	\begin{aligned}
	\|\phi D_{1,y_0} Q \psi \|_{L^2\to L^2} \leq C \left( \sup_{q\in \supp \phi} x(q) 
	+  \sup_{q\in \supp \phi} \|y(q) -y_0\| \right) \leq 2 \epsilon C.
	\end{aligned}
	\end{equation}
	Thus we may choose $\epsilon>0$ sufficiently small such that
	\begin{equation}
	\|\phi  D_{1,y_0} u\|_{L^2} \leq q \cdot \|D_{y_0} u\|_{L^2}, \;\; 
	\text{for}\;\; q<1.
	\end{equation}
Then the following inequalities hold for $u\in C^\infty_0 (\R_+\times 
	\R^b, H^\infty)$,
	\begin{equation}
	\begin{aligned}
	\|(D_{y_0} + \phi D_{1,y_0})u\|_{L^2} &\leq \|D_{y_0} u\|_{L^2} + q \|D_{y_0} 
	u\|_{L^2}\\ &\leq (1+q)\cdot \|D_{y_0} u\|_{L^2}.
	\end{aligned}
	\end{equation}
	On the other hand
	\begin{equation}
	\begin{aligned}
	\|D_{y_0} u\|_{L^2} &\leq \|(D_{y_0}+\phi D_{1,y_0})u\|_{L^2} + \|\phi D_{1,y_0} 
	u\|_{L^2}\\
	&\leq \|(D_{y_0}+\phi D_{1,y_0})u\|_{L^2} + q\|D_{y_0} u\|_{L^2}. \\
	\Rightarrow \ \|D_{y_0} u\|_{L^2} &\leq (1-q) ^{-1}\|(D_{y_0}+\phi D_{1,y_0})u\|_{L^2}.
	\end{aligned}
	\end{equation}
	Thus the graph-norms of $D_{y_0}$ and $(D_{y_0}+\phi D_{1,y_0})$ are equivalent and 
	hence their minimal domains coincide. Same statement holds for the maximal as well as the localized domains. Thus we have the following equalities. 
	\begin{equation}\label{a}
	\begin{split}
	 &\sD_{\rm comp}(D_{y_0, \, \min})  = \sD_{\rm comp}((D_{y_0} + \phi D_{1,y_0})_{\min}), \\
	 &\sD_{\rm comp}(D_{y_0, \, \max})  = \sD_{\rm comp}((D_{y_0} + \phi D_{1,y_0})_{\max}).
	 \end{split}
	\end{equation}
	The equalities continue to hold for a cutoff function 
	$\phi \in C_0^\infty((0,\infty)\times \R^b)$ such that for some $x_0 > \epsilon$,
	$\supp\phi \subset (x_0 - \epsilon, x_0 + \epsilon) \times B_\epsilon(y_0)$ by
	a similar argument. \medskip
	
	\textbf{Step 2:} We now prove the following inclusion
	\begin{equation}\label{c}
	\sD_{\rm comp}((D_{y_0} + \phi D_{1,y_0})_{\min}) 
	\subseteq \sD_{\rm comp}((D_{y_0} + D_{1,y_0})_{\min}).
	\end{equation} 
	Indeed, for any $u\in \sD_{\min}(D_{y_0} + \phi D_{1,y_0})$ there exists $(u_n)\subset 
	C^\infty_0((0,\infty)\times \R^b, H^\infty)$ converging to $u$ in the graph norm of 
	$(D_{y_0} + \phi D_{1,y_0})$. By \eqref{a} and Corollary \ref{Corollaryequalitymaxmin},
	$(u_n)$ converges to $u$ in $W^{1,1}$. Hence, using continuity of 
	$D_{1,y_0}: W^{1,1}\to L^2$ we conclude
	\begin{equation*}
	\begin{aligned}
	(D_{y_0}+D_{1,y_0})u_n &= (D_{y_0} + \phi D_{1,y_0})u_n + (1-\phi)D_{1,y_0} u_n\\ 
	&\stackrel{L^2}{\rightarrow} (D_{y_0}+\omega D_{1,y_0})u + (1-\phi)D_{1,y_0} u = (D_{y_0} + 
	D_{1,y_0})u.
	\end{aligned}
	\end{equation*}	
	Hence $u\in \sD_{\rm comp}((D_{y_0} + D_{1,y_0})_{\min})$ and \eqref{c} follows. 
	\medskip

\textbf{Step 3:} Consider now $u \in \sD_{\rm comp}(P_{\max / \min})$. Due to compact 
	support there exist finitely many points $\{(x_1,y_1), \ldots, (x_N,y_N)\} 
	\subset \R_+ \times \R^b$ and smooth cutoff functions 
	$\{\psi_1, \ldots, \psi_N\} \subset C_0^\infty([0,\infty)\times \R^b)$ such that 
	\begin{equation*}
	u = \sum_{j=1}^N \psi_j u, \quad \supp (\psi_j u) \subset 
	\left(\left(x_j - \frac{\epsilon}{2}, x_j + \frac{\epsilon}{2}\right)\cap [0, \epsilon) \right)
	\times B_{\frac{\epsilon}{2}}(y_j).
	\end{equation*}
	The maximal and minimal domains are stable under multiplication with cutoff functions 
	and hence each $\psi_j u \in \sD_{\rm comp}(P_{\max / \min})$.
	Consider for each $j=1, \ldots, N$ a cutoff function 
	$\phi_j \in C_0^\infty([0,\infty)\times \R^b)$ such that 
	$\supp\phi_j \subset ((x_j - \epsilon, x_j + \epsilon)\cap [0,\epsilon)) \times B_\epsilon(y_j)$
	and $\phi_j \restriction \supp (\psi_j u) \equiv 1$. Then as distributions
	\begin{equation*}
	P (\psi_j u) = (D_{y_j} + D_{1,y_j}) \psi_j u = (D_{y_j} + \phi_j D_{1,y_j}) \psi_j u.
	\end{equation*}
	We conclude $\psi_j u \in \sD_{\rm comp}((D_{y_j} + \phi_j D_{1,y_j})_{\max / \min})$.
	In view of \eqref{a} and Corollary \ref{Corollaryequalitymaxmin} we find
     \begin{equation}\label{d}
	 \sD_{\rm comp}(P_{\min}) \subseteq
	  \sD_{\rm comp}(P_{\max}) \subseteq 
	  W^{1,1}_{\rm comp}(\R_+ \times \R^b, H).
	\end{equation}

	\textbf{Step 4:} The statement now follows from a sequence of inclusions
\begin{equation}
	\begin{aligned}
	W^{1,1}_{\rm comp} &\ = \sD_{\rm comp}(D_{y_0, \, \min}) \\ &\stackrel{\eqref{a}}{=}
	\sD_{\rm comp}((D_{y_0} + \phi D_{1,y_0})_{\min}) 
	\stackrel{\eqref{c}}{\subseteq} \sD_{\rm comp}(P_{\min})\\
	&\ \subseteq \sD_{\rm comp}(P_{\max}) \stackrel{\eqref{d}}{=} W^{1,1}_{\rm comp}.
	\end{aligned}
	\end{equation}
	The first equality is due to Corollary \ref{Corollaryequalitymaxmin}.
	Hence all inclusions are in fact equalities and the statement follows.
\end{proof}

\section{Minimal domain of a Laplace Operator on an abstract edge}

Definition \ref{minmax} extends to define the notion of minimal 
and maximal domain for the squares $D_{y_0}^2$ and $D^2$ of the generalized
Dirac operators. Their localized versions are defined as in \eqref{localized-domains}.
In this section, we discuss the minimal and maximal domains of $D_{y_0}^2$ and $D^2$ 
by repeating the arguments of \S \ref{min-domain-Dirac-section} with appropriate changes. 
\medskip

We also note as in Remark \ref{different-sobolev-scales},
the interpolation scales $H^s(S(y))$ and $H^s(S(y_0))$ coincide for $0\leq s \leq 1$, but 
a priori may differ for $s>1$. While this was sufficient for the discussion of the domain of $D$ in the previous 
section, it is insufficient for the discussion of the domain of $D^2$.
Hence, within the scope of this section we pose the following

\begin{assumption}\label{assume-equal-interpolation-scales}
The interpolation scales $H^s(S(y))$ are independent of $y\in \R^b$ for $0 \leq s \leq 2$,
in which case we write $H^s\equiv H^s(S(y))$.
\end{assumption}

The following result follows by repeating the arguments of Lemma \ref{LemmaDcompmaxmin}
and Corollary \ref{Corollaryequalitymaxmin} ad verbatim, where
$D_{y_0}$ is replaced by $D_{y_0}^2$, $W^{1,1}$ by $W^{2,2}$ and $Q$ by $Q^2$.
These changes do not affect the overall argument.

\begin{proposition}\label{Corollaryequalitymaxmin-2}
	$\sD_{\rm comp}(D^2_{y_0, \, \min}) = \sD_{\rm comp}(D^2_{y_0, \, \max}) = W^{2,2}_{\rm 
	comp}(\R_+\times \R^b,H)$ .
\end{proposition}

Now we want to extend this statement to a perturbation of $D^2_{y_0}$
\begin{equation}
G = -\b_x^2 + X^{-2}\ S(y) \ (S(y)+1) + T^2 + W 
:=D^2_{y_0} + R_{y_0}
\end{equation}
where $W:W^{2,2}(\R_+ \times \R^b, H) \to 
W^{0,1}(\R_+ \times \R^b, H)$ is a bounded linear operator, preserving
compact supports, and is referred to as a higher order term.

\begin{theorem}\label{TheoremDcompPmimPmax-2}
Assume in addition to the spectral Witt condition \eqref{Witt-condition} that 
\begin{equation}\label{smooth-family-2}
\partial_y S(y) \circ (|S(y_0)|+1)^{-1}, \quad (|S(y_0)|+1) \circ \partial_y S(y) \circ (|S(y_0)|+1)^{-2}
\end{equation}
are bounded operators on $H$ for any $y, y_0\in \R^b$.
Then 
\begin{equation}
\sD_{\rm comp}(G_{\min}) = \sD_{\rm comp}(G_{\max}) = 
W^{2,2}_{\rm comp}(\R_+ \times \R^b, H).
\end{equation}
\end{theorem}

\begin{proof} The assumption \eqref{smooth-family-2} translates into the condition that
for $A=|S|+\frac{1}{2}$ 
\begin{equation}\label{smooth-family-3}
\partial_y A^2(y) \circ (|S(y_0)|+1)^{-2}
\end{equation}
is bounded. From there we proceed exactly as in Theorem \ref{TheoremDcompPmimPmax}.
\medskip

Consider $u\in C_0^\infty((0,\infty)\times \R^b, H^\infty)$ and 
smooth cutoff functions $\phi, \psi \in C_0^\infty([0,\infty)\times \R^b)$ 
taking values in $[0,1]$, such that $\supp\phi \subset [0,\epsilon) 
\times B_\epsilon(y_0)$ and $\psi \restriction \supp u \equiv 1$.
We compute using the analogue of \eqref{QD} for $D^2_{y_0}$
	\begin{equation}
	\begin{aligned}
		\|\phi R_{y_0} u\|_{L^2} &= \|\phi R_{y_0} Q^2 D^2_{y_0} u\|_{L^2}
		= \|\phi R_{y_0} Q^2 \psi D^2_{y_0} u\|_{L^2}\\
		&\leq \|\phi R_{y_0} Q^2 \psi\|_{L^2\to L^2} \cdot \|D^2_{y_0} 
		u\|_{L^2}.
	\end{aligned}
	\end{equation}
	In order to estimate the norm of $\phi R_{y_0} Q^2 \psi$, note that
	\begin{equation}
	\begin{aligned}
	R_{y_0} Q ^2 & = X^{-2} \left(A^2(y)-A^2(y_0)\right) Q^2
	 + \left(T^2-T^2_{y_0}\right) Q^2 +  W Q^2
	\\ & = X^{-2} (y-y_0) \int_0^1 \frac{\partial A^2}{\partial t}\left(y_0 + 
	t(y-y_0)\right)dt\ Q^2
	\\ & + (y-y_0) \int_0^1 \frac{\partial T^2}{\partial t}\left(y_0 + 
	t(y-y_0)\right)dt\ Q^2 + WQ^2.
	\end{aligned}
	\end{equation}
In view of \eqref{smooth-family-3} and boundedness of 
	the higher order term $W:W^{2,2} \to W^{0,1}$ we conclude from 
	Theorem \ref{TheoremParametrix-2} that
	\begin{equation*}
	X^{-2} \frac{\partial A^2}{\partial t}\left(y_0 + t(y-y_0)\right) Q^2 \psi, 
	\quad \frac{\partial T^2}{\partial t}\left(y_0 + t(y-y_0)\right) \, Q^2 \psi, 
	\quad X^{-1} W \, Q^2 \psi
	\end{equation*}
	are bounded operators on $L^2(\R_+ \times \R^b, H)$ with bound uniform in $t\in [0,1]$
	and $\psi$. Hence we conclude for some unform constant $C>0$
	\begin{equation}
	\begin{aligned}
	\|\phi R_{y_0} Q^2 \psi \|_{L^2\to L^2} \leq C \left( \sup_{q\in \supp \phi} x(q) 
	+  \sup_{q\in \supp \phi} \|y(q) -y_0\| \right) \leq 2 \epsilon C.
	\end{aligned}
	\end{equation}
	Thus we may choose $\epsilon>0$ sufficiently small such that
	\begin{equation}
	\|\phi  R_{y_0} u\|_{L^2} \leq q \cdot \|D^2_{y_0} u\|_{L^2}, \;\; 
	\text{for}\;\; q<1.
	\end{equation}
Then the following inequalities hold for $u\in C^\infty_0 (\R_+\times 
	\R^b, H^\infty)$,
	\begin{equation}
	\begin{aligned}
	\|(D^2_{y_0} + \phi R_{y_0})u\|_{L^2} &\leq \|D^2_{y_0} u\|_{L^2} + q \|D^2_{y_0} 
	u\|_{L^2}\\ &\leq (1+q)\cdot \|D^2_{y_0} u\|_{L^2}.
	\end{aligned}
	\end{equation}
	On the other hand
	\begin{equation}
	\begin{aligned}
	\|D^2_{y_0} u\|_{L^2} &\leq \|(D^2_{y_0}+\phi R_{y_0})u\|_{L^2} + \|\phi R_{y_0} 
	u\|_{L^2}\\
	&\leq \|(D^2_{y_0}+\phi R_{y_0})u\|_{L^2} + q\|D^2_{y_0} u\|_{L^2}. \\
	\Rightarrow \ \|D^2_{y_0} u\|_{L^2} &\leq (1-q) ^{-1}\|(D^2_{y_0}+\phi R_{y_0})u\|_{L^2}.
	\end{aligned}
	\end{equation}
	Thus the graph-norms of $D^2_{y_0}$ and $(D^2_{y_0}+\phi R_{y_0})$ are equivalent and 
	hence their minimal domains coincide. Same statement holds for the maximal as well as the localized domains. Thus we have the following equalities. 
	\begin{equation}\label{a-2}
	\begin{split}
	 &\sD_{\rm comp}(D^2_{y_0, \, \min})  = \sD_{\rm comp}((D^2_{y_0} + \phi R_{y_0})_{\min}), \\
	 &\sD_{\rm comp}(D^2_{y_0, \, \max})  = \sD_{\rm comp}((D^2_{y_0} + \phi R_{y_0})_{\max}).
	 \end{split}
	\end{equation}
The equalities continue to hold for a cutoff function 
	$\phi \in C_0^\infty((0,\infty)\times \R^b)$ such that for some $x_0 > \epsilon$,
	$\supp\phi \subset (x_0 - \epsilon, x_0 + \epsilon) \times B_\epsilon(y_0)$ by
	a similar argument. \medskip
	
	From there we may repeat the arguments of the proof of Theorem 
	\ref{TheoremDcompPmimPmax} ad verbatim, replacing $D_{y_0}$ by $D^2_{y_0}$,
	$D_{1,y_0}$ by $R_{y_0}$, $P$ by $G$, $W^{1,1}$ by $W^{2,2}$. These replacements
	do not affect the overall argument. 
\end{proof}

\section{Domains of Dirac and Laplace Operators on a Stratified Space}

Consider a compact stratified space $M_k$ of depth $k\in \N$ with an iterated 
cone-edge metric $g_k$. Each singular stratum $B$ of $M_k$ admits an open 
neighbourhood $\cU \subset M_k$ with local coordinates $y$ and a defining 
function $x_k$ such that
\begin{equation}\label{iterated-edge-metric-last}
g|_{\cU} = dx_k^2 + x_k^2 \ g_{k-1}(x_k, y) + g_B(y) + h =: \overline{g} + h,
\end{equation}
where $g_{k-1}(x_k,y)$ is a smooth family of iterated cone-edge metrics on a 
compact stratified space $M_{k-1}$ of lower depth and $h$ is a higher order
symmetric $2$-tensor, smooth on the resolution $\widetilde{\cU}$ with $|h|_{\overline{g}} = O(x_k)$
as $x_k \to 0$. 

The associated Sobolev-spaces are defined in Definition \ref{Sobolev-spaces}.
Recall, their elements take values in the vector bundle $E$, which denotes the exterior algebra of the 
incomplete edge cotangent bundle $\Lambda^* {}^{ie}T^*\cU$ in case of the 
Gauss--Bonnet operator, 
and the spinor bundle in case of the spin Dirac operator. We usually omit $E$ from the notation. 
We introduce here the localized versions of the Sobolev spaces ($s\in \N$)
\begin{equation}
\begin{aligned}
\sH^{s,\delta}_{e,{\rm comp}} :=\{\phi \cdot u \ | \ \phi \in C^\infty_0(\widetilde{\cU}), 
u\in \sH^{s,\delta}_{e}   \}.
\end{aligned}
\end{equation}

Consider the unitary transformation $\Phi$ in \eqref{unitary-trafo}, \cf 
\cite[(5.10)]{BS2},
which maps $L^2(\cU,E,\overline{g})$ to $L^2(\cU, E, \overline{g}_{\textup{prod}})$, where 
we recall $\overline{g}$ from \eqref{iterated-edge-metric-last} and set $\overline{g}_{\textup{prod}}
:= dx_k^2 + g_{k-1}(x_k, y) + g_B(y)$. The spaces $\sH^{*,*}_{e,{\rm comp}}$ 
with compact support in $\cU$ may be defined with respect
to $\overline{g}$ and $\overline{g}_{\textup{prod}}$. We indicate the choice 
of the metric when necessary,
\eg $\sH^{*,*}_{e,{\rm comp}}(M_k,\overline{g}_{\textup{prod}}), L^2_{{\rm 
comp}}
(M_k,\overline{g}_{\textup{prod}})$, and do not specify the metric 
when the statement holds for both choices. Note 
\begin{align*}
\sH^{*,*}_{e,{\rm comp}}(M_k,\overline{g}_{\textup{prod}}) = 
\Phi \sH^{*,*}_{e,{\rm comp}}(M_k,\overline{g}).
\end{align*}
Whenever we use the Sobolev spaces 
$\sH^{*,*}_{e}(M_k)$ or $L^2(M_k)$ without compact support in the open interior of $M_k$, 
we use the iterated cone-edge metric $g_k$ in the definition of the 
$L^2$-structure.

\begin{remark}We write $L^2_{\rm comp}:= 
	\sH^{0,0}_{e,{\rm comp}}$ and denote by $\rho_k$ a smooth
	function on the resolution $\widetilde{M}_k$, nowhere vanishing
	in its open interior, and vanishing to first order at each 
	boundary face of $\widetilde{M}_k$. Iteratively, $\rho_k = x_k \rho_{k-1}$. Then 
\begin{equation}\label{Sobolev-spaces-versions}
\begin{split}
\sH^{1,1}_{e,{\rm comp}} &= \rho_k \ \{u \in L^2_{\rm comp} \ | \ 
\rho_k \partial_x u,  \rho_k \partial_y u, \mathcal{V}_{e, k-1}(M_{k-1}) u \in L^2_{\rm comp} \ \} \\
&= \{u \in L^2_{\rm comp} \ | \ 
\frac{u}{\rho_k}, \partial_x u, \partial_y u, \rho_k^{-1}\mathcal{V}_{e, k-1}(M_{k-1}) u \in L^2_{\rm comp} \ \}.
\end{split}
\end{equation}
Here, the first equality in \eqref{Sobolev-spaces-versions}
follows by Definition \ref{Sobolev-spaces}, 
once we recall from \eqref{edge-vector-field-depth-k} the following
iterative structure of edge vector fields
\begin{equation}
\mathcal{V}_{e,k}\restriction \widetilde{\cU} = C^\infty(\widetilde{\cU})\textup{- span}\, 
\{\rho_k \partial_x, \rho_k\partial_y, 
\mathcal{V}_{e, k-1}(M_{k-1})\}.
\end{equation}
The second equality in \eqref{Sobolev-spaces-versions} is now straightforward.
Similarly
\begin{equation}\label{Sobolev-spaces-versions-2}
\begin{split}
\sH^{2,2}_{e,{\rm comp}} &= \rho^2_k \ \{u \in L^2_{\rm comp} \ | \ 
\{\rho_k \partial_x, \rho_k \partial_y, \mathcal{V}_{e, k-1}(M_{k-1})\}^j \, u \in L^2_{\rm comp}, \ j=1,2 \} \\
&= \{u \in L^2_{\rm comp} \ | \ 
\{\rho_k^{-1}, \partial_x, \partial_y, \rho_k^{-1}\mathcal{V}_{e, k-1}(M_{k-1})\}^j \, u \in L^2_{\rm comp}, \ 
j=1,2 \}.
\end{split}
\end{equation}
\end{remark}
	
The spin Dirac and the Gauss--Bonnet operators $D_k$ on $(M_k,g_k)$ admit under 
a rescaling $\Phi$ as in 
\eqref{unitary-trafo} the following form over the singular neighbourhood $\cU \subset M_k$
\begin{equation}\label{EquationA}
\Phi \circ D_k \circ \Phi^{-1}= \Gamma(\partial_{x_k} + X_k^{-1}S_{k-1}(y))+T + V,
\end{equation}
which satisfies the following iterative properties
\begin{enumerate}
	\item $S_{k-1}(y) = D_{k-1}(y)+R_{k-1}(y)$, where $D_{k-1}(y)$ is 
	a smooth family of differential operators (spin Dirac or the Gauss--Bonnet 
	operators) on $(M_{k-1}, g_{k-1}(0,y))$. The operators $S_{k-1}(y),D_{k-1}(y)$ extend
	continuously to bounded maps $\sH^{1,1}_e(M_{k-1}) \to L^2(M_{k-1})$. 
	Moreover, $R_{k-1}(y)$ extends continuously to a bounded operator on $L^2(M_{k-1})$;
	\item $x_k^{-1}V$ extends continuously to a map from $\sH^{1,1}_{e, {\rm comp}}$ to
	$L^2_{\rm comp}$;
	\item $T$ is a Dirac Operator on $B$.
\end{enumerate}

Since at this stage essential self-adjointness  
of each $S_{k-1}(y)$ and discreteness of its self-adjoint extension 
is yet to be established, we reformulate the spectral Witt condition 
\eqref{Witt-condition} in terms of quadratic forms. 
Here, we employ the notions introduced in 
Kato \cite[Chapter 6, \S 1]{Kato}. We define for any 
smooth compactly supported $u \in C^\infty_0(M_{k-1})$ using the inner product of
$L^2(M_{k-1}, g_{k-1}(0,y))$
\begin{equation}
t(S_{k-1}(y))[u] := \|S_{k-1}(y) u\|^2_{L^2}.
\end{equation}
This is the quadratic form associated to the symmetric differential operator $S_{k-1}(y)^2$,
densely defined with domain $C^\infty_0(M_{k-1})$ in the Hilbert space 
$L^2(M_{k-1}, g_{k-1}(0,y))$. The numerical range of $t(S_{k-1}(y))$ is defined by 
\begin{equation}
\Theta(S_{k-1}(y)) := \left\{ t(S_{k-1}(y))[u] \in \R \mid u \in C^\infty_0(M_{k-1}), 
\|u\|^2_{L^2} = 1\right\}.
\end{equation}

We can now reformulate the spectral Witt condition, cf. \eqref{Witt-condition}, 
as follows. 
 
\begin{definition}\label{Witt-stratified}
The operator $D_k$ on the stratified space $M_k$
satisfies the spectral Witt condition, if there exists $\delta > 0$ such that in all depths $j \leq k$ 
the numerical ranges $\Theta(S_{k-1}(y))$ are subsets
of $[1+\delta, \infty)$ for any $y\in B$.
\end{definition}

\begin{proposition}
Assume that $S_{k-1}(y)$ with domain $C^\infty_0(M_{k-1})$ in the Hilbert space 
$L^2(M_{k-1}, g_{k-1}(0,y))$ is essentially self-adjoint and its self adjoint 
realization is discrete. Then $\Theta(S_{k-1}(y)) \subset [1 + \delta, \infty)$ for some $\delta > 0$ if and 
only if $\textup{Spec} S_{k-1}(y) \cap [-1, 1] = \varnothing$.
\end{proposition}

\begin{proof}
By Kato \cite[Chapter 6, \S 4, Theorem 1.18]{Kato}, $\Theta(S_{k-1}(y))$ is a dense subset of 
$\textup{Spec} \, S_{k-1}(y)^2$. If the spectral Witt condition in the sense of Definition 
\ref{Witt-stratified} holds, this implies that $\textup{Spec} \, S_{k-1}(y)^2 \subset [1 + \delta, \infty)$
for some $\delta>0$. By discreteness this is equivalent to $\textup{Spec} S_{k-1}(y) \cap [-1, 1] = \varnothing$.

Conversely, if $\textup{Spec} \, S_{k-1}(y) \cap [-1, 1] = \varnothing$, then by discreteness of the 
spectrum, $S_{k-1}(y)^2 > 4+\delta$ for some $\delta > 0$. The spectral 
Witt condition in the sense of Definition \ref{Witt-stratified} now follows, since 
by Kato \cite[Chapter 6, \S 4, Theorem 1.18]{Kato}, 
$\Theta(S_{k-1}(y))$ is a dense subset of $\textup{Spec} \, S_{k-1}(y)^2$. 
\end{proof}

We can now prove our main result. 

\begin{theorem}
	Let $M_k$ be a compact stratified Witt space. Let $D_k$ denote 
	either the Gauss--Bonnet or the spin Dirac operator. Assume that $D_k$ 
	satisfies the 
	spectral Witt condition\footnote{In case of the Gauss--Bonnet 
	operator on a stratified Witt space this 
	can always be achieved by scaling the iterated cone-edge metric on fibers 
	accordingly.}. Then 
	$\sD_{\max}(D_k) = \sD_{\min}(D_k) = \sH^{1,1}_e(M_k)$.
\end{theorem}

\begin{proof}
We prove the result by induction on the following statement.

\begin{assumption}\label{iterative-assumption}
On any compact stratified space $M_j$ the operator $D_j$ satisfies the
following conditions near each stratum $B$:
For $y\in B$,
		$S_{j-1}(y)$ admits a unique self-adjoint extension in $L^2(M_{j-1})$
		with discrete spectrum and $\spec S_{j-1}\cap [-1, 1] = \varnothing$.
		The unique self-adjoint domain of $S_{j-1}(y)$ is given by $\sH^{1,1}_e(M_{j-1})$.
		The compositions $S_{j-1}(y) (|S_{j-1}(y_0)|+1)^{-1}$ and 
		$\partial_y S_{j-1}(y) (|S_{j-1}(y_0)|+1)^{-1}$ are bounded on $L^2(M_{j-1})$
		for $y,y_0\in B$.
\end{assumption}

These assumptions are trivially satisfied if $j=1$. Assume that 
Assumption \ref{iterative-assumption} is satisfied for $j\leq k$. 
We need to prove that Assumption \ref{iterative-assumption} is then 
satisfied for $j \leq k+1$. Let $\sD_{\rm comp}(D_k)$ denote 
elements in the maximal domain of $D_k$ with compact support in $\widetilde{\cU}$. 
Then by Theorem \ref{TheoremDcompPmimPmax}, we conclude
\begin{equation*}
	\begin{aligned}
	\Phi \sD_{\rm comp}(D_k) &\equiv \sD_{\rm comp}(\Phi \circ D_k \circ \Phi^{-1}) 
	\\ &= W^{1,1}_{\rm comp} (\R_+\times 
	\R^b,H^{\bullet}(S_{k-1}))\\
	&= \sH^{1,1}_{e,{\rm comp}}(\R_+\times \R^b)\hot L^2(M_{k-1}) 
	\cap \sH^{0,1}_{e,{\rm comp}}(\R_+\times 
	\R^b)\hot \sH^{1,1}_e(M_{k-1})\\
	&\subseteq \{u \in L^2_{\rm comp} (M_k,\overline{g}_{\textup{prod}})\ | \ 
	\frac{u}{\rho_k}, \partial_x u, \partial_y u, \rho_k^{-1}\mathcal{V}_{e, k-1}(M_{k-1}) u 
	\\ &\in L^2_{\rm comp} (M_k,\overline{g}_{\textup{prod}}) \} =\sH^{1,1}_{e,{\rm comp}}
	(M_k, \overline{g}_{\textup{prod}}) \equiv \Phi \sH^{1,1}_{e,{\rm comp}}
	(M_k, \overline{g}),
	\end{aligned}
	\end{equation*}
where we used \eqref{Sobolev-spaces-versions}
in the last line. On the other hand it is straightforward to 
check that 
\begin{equation*}
\begin{aligned}
&\Phi \sH^{1,1}_{e,{\rm comp}}
	(M_k, \overline{g}) \equiv \sH^{1,1}_{e,{\rm comp}}
	(M_k, \overline{g}_{\textup{prod}}) \\ &= \rho_k \ \{u \in L^2_{\rm comp}(M_k, \overline{g}_{\textup{prod}}) \ | \ 
\rho_k \partial_x u,  \rho_k \partial_y u, \mathcal{V}_{e, k-1}(M_{k-1}) u \in L^2_{\rm comp} (M_k, \overline{g}_{\textup{prod}}) \} \\
&\subseteq \sH^{1,1}_{e,{\rm comp}}(\R_+\times \R^b)\hot L^2(M_{k-1}) 
\cap \sH^{0,1}_{e,{\rm comp}}(\R_+\times 
\R^b)\hot \sH^{1,1}_e(M_{k-1}) \\
& = W^{1,1}_{\rm comp} (\R_+\times 
\R^b,H^{\bullet}(S_{k-1})) = \sD_{\rm comp}(\Phi \circ D_k \circ \Phi^{-1}) 
\equiv \Phi \sD_{\rm comp}(D_k).
\end{aligned}
\end{equation*}
We conclude $\sD_{\rm comp}(D_k) = \sH^{1,1}_{e,{\rm comp}}(M_k, \overline{g})$
and hence $\sD(D_k) = \sH^{1,1}_{e}(M_k)$. 
Essential self-adjointness of $D_k$ implies essential self-adjointness of $S_k$
with the domain of both given by $\sH^{1,1}_{e}(M_k)$ independently of parameters.
The domain $\sH^{1,1}_{e}(M_k)$ embeds compactly into $L^2(M_k)$ and hence
both $D_k$ and $S_k$ are discrete. 

Since $S_k$ is discrete, the spectral Witt condition of Definition 
\ref{Witt-stratified} implies 
\begin{equation}
\spec S_{k} \cap \left[-1, 1\right] = \varnothing.
\end{equation} 
The mapping properties of $(|S_{k}|+1)^{-1}$ are derived from the mapping properties 
of the model parametrix in Theorem \ref{TheoremParametrix} in the usual way and 
hence $(|S_{k}|+1)^{-1}:L^2(M_k) \to \sH^{1,1}_{e}(M_k)$ is bounded. Since 
$S_k, \partial_y S_k$ are bounded maps from $\sH^{1,1}_{e}(M_k)$
to $L^2(M_k)$ by the iterative properties of the individual operators in \eqref{EquationA},
we conclude that Assumption \ref{iterative-assumption} 
is satisfied for $j \leq k+1$ and hence holds for all $j \in \N$.
\end{proof}
	
Similar arguments apply for the Laplace operators.

\begin{corollary}
	Let $M_k$ be a compact stratified Witt space. Let $D_k$ denote 
	either the Gauss--Bonnet or the spin Dirac operator. Assume that $D_k$ 
	satisfies the 
	spectral Witt condition. Then
	$\sD_{\max}(D^2_k) = \sD_{\min}(D^2_k) = \sH^{2,2}_e(M_k)$.
\end{corollary}

\begin{proof}
We prove the result by induction. The statement is trivially satisfied if $k=0$. Assume that 
the statement holds for $(k-1)\in \N_0$. In particular, by induction hypothesis and by 
Theorem \ref{TheoremDcompPmimPmax}
\begin{equation}\label{domains-S-S-square}
\begin{split}
&H^1(S_{k-1})\equiv \sD (S_{k-1}) = \sH^{1,1}_e(M_{k-1}), \\ 
&H^2(S_{k-1}) \equiv \sD (S^2_{k-1}) = \sH^{2,2}_e(M_{k-1}).
\end{split}
\end{equation}
Since the domains $\sD (S^2_{k-1}(y))$ are independent of $y$ by the induction 
hypothesis, their interpolation scales $H^s(S_{k-1}(y))$ coincide for $0\leq s \leq 2$
and the Assumption \ref{assume-equal-interpolation-scales} is satisfied. 
The spectral Witt condition is satisfied in each depth by Theorem \ref{TheoremDcompPmimPmax}.
We need to prove the statement for $k$. Let $\sD_{\rm comp}(D^2_k)$ denote 
elements in the maximal domain of $D^2_k$ with compact support in $\widetilde{\cU}$. 
Then by Theorem \ref{TheoremDcompPmimPmax-2} and \eqref{domains-S-S-square} we conclude
\begin{equation*}
	\begin{aligned}
	\Phi \sD_{\rm comp}(D^2_k) &\equiv \sD_{\rm comp}(\Phi \circ D^2_k \circ \Phi^{-1})\\
	&= W^{2,2}_{\rm comp} (\R_+\times 
	\R^b,H^{\bullet}(S_{k-1}))\\
	&= \sH^{2,2}_{e,{\rm comp}}(\R_+\times \R^b)\hot L^2(M_{k-1}) 
	\\ &\cap \sH^{0,2}_{e,{\rm comp}}(\R_+\times 
	\R^b)\hot \sH^{1,1}_e(M_{k-1}) \\ &\cap \sH^{0,2}_{e,{\rm comp}}(\R_+\times 
	\R^b)\hot \sH^{2,2}_e(M_{k-1}) \\ &= \sH^{2,2}_{e,{\rm comp}}(M_k,\overline{g}_{\textup{prod}})
	\equiv \Phi \sH^{2,2}_{e,{\rm comp}}(M_k,\overline{g}).
	\end{aligned}
	\end{equation*}
where we used \eqref{Sobolev-spaces-versions-2}
in the last equality. The statement follows.
\end{proof}

We conclude the section with pointing out that while we cannot geometrically control the 
spectral Witt condition in case of the spin Dirac operator, for the 
Gauss--Bonnet 
operator on a stratified Witt space, we find $0 \notin \spec S_{k}$ in each iteration step, and 
can scale the spectral gap up by a simple rescaling of the metric to achieve the spectral 
Witt condition.

\appendix
\section{\relax}\label{s.TPO}
\subsection*{Notation} In this section matrices $(a_{ij})_{1\le i,j\le n}$
will often be abbreviated  $(a_{ij})_{ij}$ as long as the size $n$ is clear
from the context. Summations $\sum_{i,j,k,\ldots}$ will always denote a
\emph{finite} sum where all summation indices run independently from $1$
to $n$.

\subsection{Positivity of Matrices of Operators on Hilbert spaces}
\
The following result is based on Lance \cite[Lemma 4.3]{Lan}.

\begin{proposition}\label{p.TPO.1} 
  Let $a=(a_{ij})_{1\leq i, j \leq n}$, $b=(b_{ij})_{1\leq i,j \leq n}$ be
matrices of operators on Hilbert spaces $H_1$, $H_2$, respectively. 
\textit{I.e.,}
$a_{ij}\in \sL(H_1)$, $b_{ij}\in\sL(H_2)$. We may view $a$ as
an element of $\Mat_n(\sL(H_1))$  or of $\sL(H_1^n)$. 
Assume that $a\geq 0$ and $b\geq 0$. Then the following holds.
\begin{thmenum}
\item $(a_{ij}\otimes b_{ij})_{ij}\geq 0$ in 
$\sL((H_1\hot  H_2)^n) = \Mat_n(\sL(H_1\hot H_2))$;
		
\item $\sum_{i,j} a_{ij} \otimes b_{ij} \geq 0$ in 
		$\sL(H_1\hot H_2)$.
		
\item If $a\leq c=(c_{ij})_{ij} \in \sL(H_1^n)$, 
  $b\leq d=(d_{ij})_{ij}\in \sL(H_2^n)$ then
\begin{equation}
  (a_{ij} \otimes b_{ij})_{ij} \leq (c_{ij}\otimes d_{ij})_{ij}.
\end{equation}
\end{thmenum}
\end{proposition}

Note that for $H_1 = H_2  = \C$ this is an elementary statement about positive
semi-definite matrices.

\begin{proof}
\textup{(1)} Write $a=s^{\ast} s$, $s=(s_{ij})$, $b=t^{\ast} t$, $t=(t_{ij})$. 
Thus $a_{ij} = \sum_k s^{\ast}_{ki}s_{kj}$, $b_{ij} = \sum_k t^{\ast}_{ki} t_{kj}$,
and
\begin{equation}
    a_{ij} \otimes b_{ij} 
       = \sum_{k,l} s^{\ast}_{ki}s_{kj} \otimes t_{li}^{\ast} t_{lj}
       = \sum_{k,l} (s_{ki}\otimes t_{li})^{\ast}(s_{kj}\otimes t_{lj}).
\end{equation}
So it suffices to prove that the matrices
\begin{equation}
   \bigl\{ (s_{ki}\otimes t_{li})^{\ast} (s_{kj}\otimes t_{lj})\bigr\}_{ij} \geq 0.
\end{equation}
For fixed $k,l$ let $T_{i}:=s_{ki}\otimes t_{li}$. Then for
$\xi=(\xi_i)_{1\leq i \leq n} \in (H_1\hot  H_2)^n$ we have
\begin{equation}
  \begin{split}
  \langle (T_i^\ast T_j)_{ij} \xi, \xi \rangle 
      &= \Big\langle \Bl\sum_k T_i^\ast T_k \xi_k\Br_i, \xi \Bigr\rangle
              = \sum_{i,k} \langle T_i^\ast T_k \xi_k , \xi_i \rangle\\
      &= \sum_{i,k} \langle T_k \xi_k , T_i \xi_i \rangle 
          = \| (T_i \xi_i)_i \|^2 \geq 0.
  \end{split}
\end{equation}
So indeed the matrix $(T_i^\ast T_j)_{ij}$ is $\geq 0$. 
	
\textup{(2)}        
It suffices to show that if $(f_{ij})_{ij} := (a_{ij}\otimes b_{ij})_{ij} \geq 0$ then 
$\sum_{i,j} f_{ij} \geq 0$. Given $x\in H$ put $y_i = x$, $y = (y_i)_{1\leq i \leq n}\in H^n$. 
Then
\begin{equation}
   \begin{split}
     0\leq \langle (a_{ij})\cdot (y_i), (y_i) \rangle 
         &=\sum_i\Bigl\langle \sum_j  a_{ij}y_j,y_i\Bigr\rangle\\
	 &=\Bigl\langle \sum_{i,j} a_{ij} x, x\Bigr\rangle 
         = \Bigl\langle \Bl\sum_{i,j} a_{ij}\Br x,x\Bigr\rangle.
\end{split}
\end{equation}
	
\textup{(3)} From $c-a\geq 0$ and $d-b\geq 0$ and (1) we infer
that the matrices $\bl (c_{ij}-a_{ij})\otimes b_{ij} \br$
and $\bl c_{ij}\otimes ( d_{ij}-b_{ij}) \br$ are $\ge 0$
and hence  
\begin{equation*}
    0 \leq \bl (c_{ij}-a_{ij})\otimes b_{ij} \br + 
        \bl c_{ij}\otimes ( d_{ij}-b_{ij}) \br
    = ( c_{ij}\otimes d_{ij} )  - ( a_{ij} \otimes b_{ij}).\qedhere
\end{equation*}
\end{proof}

\begin{proposition}\label{p.TPO.2} 
Let $A,B$ be self-adjoint operators in Hilbert spaces $H_1$, $H_2$, 
respectively, and let
\begin{equation}
    A\hot  B:= \text{ closure of } A\otimes_\alg B \text{ on }
    \sD^{\infty}(A)\otimes_\alg \sD^{\infty} (B),
\end{equation}
where $\sD^{\infty}(A):= \bigcap_{s\geq 0} \sD(|A|^{s})$. 
Then $A\hot B$ is self-adjoint and 
$\sD^{\infty}(A)\otimes_\alg \sD^{\infty} (B)
  = \sD^{\infty}(A) \otimes_\alg H_2 \cap
     H_1 \otimes_\alg \sD^{\infty}(B)$.
\end{proposition}	
	
\begin{proof} It is straightforward to see that $A\otimes B$
is symmetric on $\sD^{\infty}(A)\otimes_\alg \sD^{\infty} (B)$
and hence $A\hot B$ is a symmetric closed operator. It remains
to show self-adjointness which is equivalent to the denseness of
the ranges $\ran (A\hot B\pm i I)$.

First we prove the statement for $B = I$ being the identity on $H_2$.
Then the graph norm of $A\otimes I$ on $\sD^{\infty}(A)\otimes H_2$ is the Hilbert 
space tensor norm for $\sD(A)\hot  H_2$. Hence $\sD(A\hot  I) = \sD(A)\hot  H_2$. 
The resolvent of $A\hot  I$ is obviously $(A\hot  I - \lambda I  \hot  I)^{-1} = (A-\lambda I )^{-1}\hot  I$.
Thus the denseness of $\ran (A\hot I \pm I\hot I)$ follows from the denseness
of $\ran (A\pm I)$. Hence $A\hot  I$ is self-adjoint. 
	
For general $B$ we now know that $A\hot  I$ and $I\hot  B$ 
are commuting self-adjoint operators. Hence
$(A\hot  I)\cdot (I \hot  B)$ is essentially self-adjoint on
\begin{equation}
  \sD^{\infty}(A) \otimes_\alg H_2 \bigcap H_1 \otimes_\alg \sD^{\infty}(B).
\end{equation}
It remains to see that the latter equals 
$\sD^{\infty}(A) \otimes_\alg \sD^{\infty}(B)$. Because then
$(A\otimes I)\cdot (I\otimes B) = A\otimes_\alg B$ and we conclude
the essential self-adjointness of $A\otimes_\alg B$.

To this end consider $\xi\in
 \sD^{\infty}(A) \otimes_\alg H_2 \bigcap H_1 \otimes_\alg \sD^{\infty}(B)$.
Then there exist
$x_i\in \sD^{\infty}(A)$, $y_i\in H_2$, 
$\tilde{x}_i\in H_1$, $\tilde{y}_i\in \sD^{\infty}(B)$, $i=1,\ldots n$ 
such that
\begin{equation}
    \sum_i x_i \otimes y_i = \xi = \sum_i \tilde{x}_i \otimes \tilde{y}_i,
\end{equation}
where without loss of generality we may assume that $\tilde{y}_i$ is 
orthonormal in $\sD^{\infty}(B)$. There is an obvious pairing
\begin{equation}
    \bl H_1 \otimes_\alg H_2 \br \times H_2 \to H_1,
\end{equation}
induced by the $H_2$ scalar product. Pick an index $j$. Then on the one 
hand
\begin{equation}
    \Bigl\langle\sum_i \tilde{x}_i \otimes \tilde{y}_i, (I+B^2)\tilde{y}_j 
	\Bigr\rangle = \sum_i \tilde{x}_i \langle \tilde{y}_i, \tilde{y}_j 
	\rangle_B = \tilde{x}_j,
\end{equation}
and on the other hand
\begin{equation}
  \begin{split}
    \Bigl\langle\sum_i \tilde{x}_i \otimes \tilde{y}_i, (I+B^2)\tilde{y}_j \Bigr\rangle 
        &= \Bigl\langle \sum_i x_i \otimes y_i , (I+B^2) \tilde{y}_j\Bigr\rangle\\
	&= \sum_i x_i\, \langle y_i, (I+B^2) \tilde{y}_j\rangle \in \sD^\infty(A).
   \end{split}
\end{equation}
This proves $\tilde{x}_j\in \sD^{\infty}(A)$ for any $j=1,\ldots,n$ and 
the statement follows. 
\end{proof}
	
\begin{proposition}\label{p.TPO.3} 
Let $A,C\geq 0$ be self-adjoint operators in $H_1$; $B,D\geq 0$ 
self-adjoint operators in $H_2$. If $A\leq C$, $\sD(C)\subset \sD(A)$ and 
$B\leq D$, $\sD(D)\subset\sD(B)$ then
\begin{equation}
  A\hot  B \leq C\hot  D,\quad  \sD(C\hot  D) \subset \sD(A\hot  B).
\end{equation}
\end{proposition}

\begin{proof}
The domain inclusion is clear from Proposition \ref{p.TPO.2}. To prove
the inequality, let $\sum_{i=1}^n x_i \otimes y_i
\in \sD^{\infty}(C)\otimes_\alg \sD^{\infty}(D)$ be given. Consider
the matrices $(\langle A x_i, x_j\rangle)_{ij}$, 
$(\langle C x_i, x_j\rangle)_{ij}$, 
$(\langle B y_i, y_j\rangle)_{ij}$,
and  $(\langle D y_i, y_j\rangle)_{ij}$. For complex 
numbers $\lambda_i$ we have
\begin{equation}
  \begin{split}
    \sum \overline{\lambda}_i \langle A x_i, x_j\rangle \lambda_j 
        &= \langle A \sum \lambda x_i , \lambda_i x_i\rangle \geq 0\\
 \langle A \sum \lambda x_i , \lambda_i x_i\rangle 
 	&\leq \langle C \sum \lambda x_i , \lambda_i x_i\rangle=\sum 
	\overline{\lambda}_i \langle C x_i, x_j\rangle \lambda_j.
\end{split}
\end{equation}
Thus we have the matrix inequalities
\begin{equation}
	0 \leq \left(\langle A x_i, x_j\rangle\right)_{ij} \leq 
		\left(\langle C x_i, x_j\rangle\right)_{ij}
\end{equation}
and analogously
\begin{equation}
  0 \leq \left(\langle B y_i, y_j\rangle\right)_{ij} \leq 
	\left(\langle D y_i, y_j\rangle\right)_{ij}.
\end{equation}
Proposition \ref{p.TPO.1} implies
\begin{equation*}
  \begin{split}
    0 &\leq \sum_{i,j} \langle (C-A) x_i, x_j\rangle \langle D y_i, 
        	y_j\rangle + \sum_{i,j} \langle A x_i, x_j\rangle\langle (D-B) y_i, 
	             y_j\rangle\\
    &= \sum_{i,j} \langle C x_i, x_j\rangle \langle D y_i, y_j\rangle - 
         	\langle A x_i, x_j\rangle \langle B y_i, y_j\rangle\\
    &= \Bigl\langle (C\otimes D) \sum x_i \otimes y_i, \sum x_i \otimes 
   	y_i\Bigr\rangle - \Bigl\langle (A\otimes B) \sum x_i \otimes y_i, 
        \sum x_i \otimes y_i\Bigr\rangle,
\end{split}
\end{equation*}
and hence $A\hot B \le C\hot D$.
\end{proof}

\subsection{Uniform asymptotic expansions of modified Bessel functions}

According to Olver \cite[p. 377 (7.16), (7.17)]{Olv:AAS}, we may write
for any $\mu > 0$ and $x>0$
\begin{equation}\label{full-expansion-olver}
\begin{split}
    I_\mu (\mu x)  &= \frac{1}{\sqrt{2\pi\mu}}\cdot 
	\frac{e^{\mu \cdot\eta(x)}}{(1+x^2)^{1/4}} 
	\left(\sum_{j=0}^{n-1} \frac{U_j(p(x))}{\mu^j} + \eta_{n,1}(\mu, x)\right) 
	\frac{1}{1+\eta_{n,1}(\mu, \infty)}, \\
	K_\mu (\mu x) &= \sqrt{\frac{2\pi}{\mu}}\cdot 
	\frac{e^{-\mu \cdot\eta(x)}}{(1+x^2)^{1/4}} 
	\left(\sum_{j=0}^{n-1} (-1)^j \frac{U_j(p(x))}{\mu^j} + \eta_{n,2}(\mu, 
	x)\right)
\end{split}
\end{equation}
where $p(x)= \sqrt{1+x^2}$, $\eta(x) = p(x) +\ln \frac{x}{1+p(x)}$
and $U_j(p)$ are iteratively defined polynomials in $p$ with $U_0 \equiv 1$. 
By Olver \cite[p. 377 (7.14), (7.15)]{Olv:AAS}, the error terms $\eta_{n,1}$ and 
$\eta_{n,2}$ admit the following bounds
\begin{equation}\label{variation-bounds}
\begin{split}
    |\eta_{n,1}(\mu, x)| &\leq 2 \exp \left(\frac{2\mathscr{V}_{(1,p(x))}(U_1)}{\mu}\right)
    \frac{\mathscr{V}_{(1,p(x))}(U_n)}{\mu^n}, \\ 
    |\eta_{n,2}(\mu, x)| &\leq 2 \exp \left(\frac{2\mathscr{V}_{(0,p(x))}(U_1)}{\mu}\right)
    \frac{\mathscr{V}_{(0,p(x))}(U_n)}{\mu^n}
\end{split}
\end{equation}
where $\mathscr{V}_{(a,b)}(f)$ denotes the total variation of a differentiable function $f$ 
along an interval $(a,b)$. In case of complex-valued arguments $x$, one takes here the 
variation along $\eta(x)$-progressive paths. However, here $x, p(x), \eta(x)$ are
all real-valued, and $\eta(x)$ is monotonously increasing as $x\to \infty$ by \eqref{eta-grow}.

Since $p((0,\infty)) = (0,1)$, we may take in \eqref{variation-bounds} 
variation over $(0,1)$ for both error terms. Since for any $j\in \N$ the total variations
$\mathscr{V}_{(0,1)}(U_j)$ are taken along finite 
paths and since $U_j$ are polynomials, we conclude that for any $n\in \N_0$
\begin{equation}\label{variation-bounds-2}
 \eta_{n,1}(\mu, x) = O(\mu^{-n}),  \quad
 \eta_{n,2}(\mu, x) = O(\mu^{-n}), 
 \ \textup{as} \ \mu \to \infty. 
\end{equation}
uniformly in $x\in (0,\infty)$. Hence the expansions 
\eqref{full-expansion-olver} are uniform in $x\in (0,\infty)$ as well.

\bibliography{local}
\bibliographystyle{amsalpha-lmp}

\end{document}